\documentclass[11pt,a4paper]{article}
\usepackage{amsfonts}
\usepackage{amsmath,amssymb}
\usepackage{mathrsfs}
\usepackage{hyperref}
\usepackage{graphicx}
\usepackage{amsthm}
\usepackage{color}
\usepackage{dsfont}
\usepackage{bbm, dsfont}
\usepackage{cite}

\newtheorem{theorem}{Theorem}[section]

\newtheorem{lemma}[theorem]{Lemma}

\newtheorem{definition}[theorem]{Definition}
\newtheorem{remark}{Remark}
\numberwithin{equation}{section}\allowdisplaybreaks

\begin{document}

\title{\large\bf  Dispersive decay for the mass-critical generalized Korteweg-de Vries equation and generalized Zakharov-Kuznetsov equations}
\author{\normalsize Minjie  Shan$^{a,*}$\\
\footnotesize
\it $^a$ College of Science, Minzu University of China, Beijing 100081, P.R. China. \\  \\ \footnotesize
\it E-mail:  smj@muc.edu.cn   \\
}
\date{} \maketitle
\thispagestyle{empty}

\begin{abstract}
In this paper,  we discuss  pointwise decay estimate for the solution to the mass-critical generalized Korteweg-de Vries (gKdV) equation with initial data $u_0\in H^{1/2}(\mathbb{R})$. It is showed that nonlinear solution enjoys the same decay rate as linear one. Moreover, we also quantify the decay for solutions to the generalized Zakharov-Kuznetsov equation which is a natural multi-dimensional extension of the gKdV equation. We obtain some decay estimates for nonlinear solutions to generalized Zakharov-Kuznetsov equations with small initial data in $H^2(\mathbb{R}^d)$.
\end{abstract}
\section{Introduction}
We consider the initial value problem (IVP) for the mass-critical generalized KdV (gKdV)
\begin{equation}
	\left\{
	\begin{aligned}
		&\partial_{t}u +\partial_{x}^{3} u \pm \partial_{x} (u^5) = 0, \\
		&u(0,x)=u_0(x),\ \ \ x\in \mathbb{R}, \ t\in\mathbb{R}. \label{gKdV} \\
	\end{aligned}
	\right.
\end{equation}
where $u(t, x)$ is a real-valued function. With the plus sign, this is a focusing gKdV equation; with the
minus sign, it is defocusing.

The mass 
$$M(u)=\int_{\mathbb{R}}u^2(t, x) dx$$
and the energy
$$E(u)=\frac{1}{2}\int_{\mathbb{R}}(\partial_x u)^2 (t, x) dx\mp \frac{1}{6}\int_{\mathbb{R}}  u^6(t, x) dx$$
 are  conserved by the flow of \eqref{gKdV}. 

This equation enjoys the following scaling symmetry:
\begin{align*}
u(t, x) \longmapsto
u_{\lambda}(t, x) = \lambda^{1/2} u ( \lambda^{3}t, \lambda x)  \hspace{5mm} \text{for} \hspace{2mm} \lambda>0,
\end{align*}
in the sense that if $u(t, x)$ solves \eqref{gKdV} then so does  $u_{\lambda}(t, x)$ with initial datum
$u_{\lambda}(0, x)=\lambda^{1/2} u ( 0, \lambda x).$ It is easy to see that
\begin{align*}M\big(u_{\lambda}(t)\big)=M\big(u(t)\big) \hspace{8mm} \text{for all} \hspace{2mm} \lambda>0.\end{align*}
Hence, we call \eqref{gKdV} the mass-critical generalized KdV equation.

Kenig, Ponce and Vega \cite{KPV93}  proved  that \eqref{gKdV} is global well-posed and scattering in $L^2(\mathbb{R})$  under a smallness condition for initial data.

\begin{theorem}[Kenig-Ponce-Vega \cite{KPV93}, 1993] \label{KPV-gKdV-Globalscattering}
There exists $0<\delta\ll1$ such that for any $u_0 \in L^2(\mathbb{R})$  with
 $$\|u_0\|_{ L^2(\mathbb{R})}<\delta,$$
there exists a unique strong solution $u(t)$ of the IVP \eqref{gKdV} satisfying
	\begin{align}
	& \hspace{13mm} u\in C\big(\mathbb{R}; L^2(\mathbb{R}) \big)\cap L^{\infty}\big(\mathbb{R}; L^2(\mathbb{R}) \big),\label{KPV-gKdV-GlobalscatteringA1}  \\
	&   \|u\|_{L^{5}_{x}L^{10}_{t}}=\left(\int_{\mathbb{R}} \left(\int_{\mathbb{R}} |u(t, x)|^{10} dt \right)^{1/2} dx \right)^{1/5}< \infty,\label{KPV-gKdV-GlobalscatteringA2}\\
	&\hspace{25mm}  	\left\|\partial_x u \right\|_{L^{\infty}_{x}L^{2}_{t}}< \infty,\label{KPV-gKdV-GlobalscatteringA3}
	\end{align}
and the global strong solution $u(t)$ scatters in $L^2(\mathbb{R})$ to a solution of the linear KdV equation as $t\to \pm \infty$, i.e., there exist unique $u_0^{\pm}\in L^2(\mathbb{R})$ such that
	\begin{align}
		\lim_{t\to \pm \infty}\left\|u(t)- e^{-t\partial^3_x}u_0^{\pm}\right\|_{L^{2}_{x}}=0. \label{KPV-gKdV-GlobalscatteringA4}
	\end{align}

Moreover, if  $u_0 \in H^s(\mathbb{R})$  with $s>0$ and $\|u_0\|_{ L^2(\mathbb{R})}<\delta$, then the solution to \eqref{gKdV} satisfies
\begin{align}
u\in C\big(\mathbb{R}; H^s(\mathbb{R}) \big)\cap L^{\infty}\big(\mathbb{R}; H^s(\mathbb{R}) \big)\label{KPV-gKdV-GlobalscatteringA5} 
	\end{align}
and
\begin{align}
\left\|D_x^s\partial_x u \right\|_{L^{\infty}_{x}L^{2}_{t}}< \infty. \label{KPV-gKdV-GlobalscatteringA6}
	\end{align}
\end{theorem}

By using concentration compactness method, Dodson  \cite{Dodson2017} showed that the defocusing  mass-critical generalized KdV is globally well-posed and scattering for arbitrary initial data  $u_0 \in L^2(\mathbb{R})$.  Furthermore, the global solution satisfies the following spacetime bounds
	\begin{align}
\| u \|_{L^{5}_{x}L^{10}_{t}(\mathbb{R}\times\mathbb{R})}\leq C\big(M(u_0)\big).\label{gKdV-Dodson}
	\end{align}
As there is no local well-posedness in $H^s$ for any $s<0$, this result is sharp,  see \cite{CCT03}.  Besides, for the mass-critical  focusing generalized KdV equation,  the concentration phenomenon of blow up solutions were studied in \cite{KPV2000, KSV12}. We refer to an series of impressive works \cite{MM00, MM01, MM02, MM02a, MM02b, MMR14, MMR15, Merle01} for more information on focusing solitons and on other blow-up solutions.

The scattering results mentioned above indicate that the long-time asymptotic development of solution to the nonlinear equation \eqref{gKdV} behaves like a solution to the linear KdV equation.  The spacetime bounds \eqref{KPV-gKdV-GlobalscatteringA2} and  \eqref{gKdV-Dodson} also hold true for the linear KdV equation, see  \eqref{gKdVL(XT)Norm1} with $\theta=\frac{4}{5}$. In fact, the spacetime inequality  with respect to $L^{5}_{x}L^{10}_{t}$ norm can be seen as another version of Strichartz estimate \cite{Stri77}  in the sense of exchanging the variables  $x$ and $t$. The motivation is the Kato smoothing effect, see \eqref{KatoSmootEff1}. By using a mixed spacetime estimate with first the $L^2$-norm in time $t$ and then the $L^{\infty}$-norm in space $x$, one can gain more smoothing properties which are crucial for handling the nonlinear term with derivative.

Kenig, Ponce and Vega  \cite{KPV89} first observed this principle in the smoothing effect of Strichartz type,  they showed the following estimates
\begin{align}
		\left\|D_x^{\frac{\theta\alpha}{2}}e^{-t\partial^3_x} u_0\right\|_{L_t^qL^p_{x}} &\lesssim
		\|u_0\|_{L^2_{x}} \nonumber
	\end{align}
with $(\theta, \alpha)\in [0, 1]\times [0, 1/2]$ and $(q,p)=(\frac{6}{\theta(\alpha+1)}, \frac{2}{1-\theta})$. The Strichartz estimates for KdV which do not involve gain of derivatives 
\begin{align}
		\left\|e^{-t\partial^3_x} u_0\right\|_{L_t^qL^p_{x}} &\lesssim
		\|u_0\|_{L^2_{x}} \nonumber
	\end{align}
were given previously in \cite{GiVe85, Mar81, Pe85}  with $\frac{6}{q}+\frac{2}{p}=1$, $2\leq p \leq \infty$. The integrability in time  shows that the solutions of the (non)linear KdV equation disperse. It is the emergence of dispersive effect that weakens  the impact of nonlinearity, thus the nonlinear effects become asymptotically negligible. This provides us an intuitive explanation for \eqref{KPV-gKdV-GlobalscatteringA4}. Thereby, we may expect the $L^{\infty}_{x}$ or other $L^p_{x}$ norm of the solution to go to zero as $t \to \pm \infty$. In particular, one natural question to ask is whether solutions to the nonlinear equation exhibit the same dispersive decay as solutions to the corresponding linear equation.

Strichartz estimates can be derived from the classical dispersive estimates by using Hardy-Littlewood-Sobolev inequality and a standard $TT^*$ argument. For the KdV equation, the dispersive estimates read as
\begin{align}
		\left\|e^{-t\partial^3_x}u_0\right\|_{L^r} \lesssim  |t|^{-\frac{1}{3}(1-\frac{2}{r}) } \left\|u_0\right\|_{L^{r'}}  \hspace{5mm}  \text{for } \hspace{1mm} t\neq0  \hspace{1mm} \text{and}\hspace{1mm}   2\leq r \leq \infty.    \label{jcKdV-DisDecay}
	\end{align}
Decay estimates are extremely useful in the study of the long-time asymptotic behaviour of nonlinear dispersive equations of various types, such as KdV, nonlinear Schr\"odinger and nonlinear wave equations.

In recent years, decay estimates pointwise in time for nonlinear dispersive equations have been  widely studied. Note that  \eqref{jcKdV-DisDecay} is not a-priori obvious for nonlinear solution provided that \eqref{KPV-gKdV-GlobalscatteringA4} holds ture. Because, it is not known whether $u_0$ is in $L^{r'}$, and certainly the decay rate is also not known. Actually, there is a close connection between the decay estimate for nonlinear solutions and the asymptotic convergence rate in \eqref{KPV-gKdV-GlobalscatteringA4}. Lin and Strauss \cite{LinStr78} established the decay of the $L^{\infty}$-norm of solutions to the 3D nonlinear Schr\"odinger equation (NLS) by using Morawetz estimate. Grillakis and Machedon \cite{GriMach13} showed the decay estimate for the cubic Hartree equation. Decay estimates can also be derived by using the vector field methods and commutator type estimates, see \cite{Klai85, KlaiPon83, Shatah82, HayTsu86} for more details. Initially, these results were obtained under strong regularity and decay hypotheses; see, for example, \cite{LinStr78, Klai85, KlaiPon83, MorStr72, Seg83} as well as the references therein. Fan and Zhao \cite{FZ21},  Guo, Huang and Song \cite{GHS23} respectively proved $L^{\infty}$ dispersive decay for the energy-critical NLS with initial data in $H^3(\mathbb{R}^3)$. Fan, Staffilani and Zhao \cite{FSZ24} discussed quantitative decay estimates for the cubic NLS with initial data in $H^1(\mathbb{R}^3)$ and with random initial data. In \cite{FZ23}, Fan and Zhao showed the existence of a special solution to defocusing cubic NLS, which lives in $H^s(\mathbb{R}^3)$ for all $s>0$, but scatters to a linear solution in a very slow way. Their construction of initial data was inspired by concentration compactness method. The regularity had been significantly reduced by Fan, Killip, Visan  and Zhao \cite{FKVZ24} very recently. They proved dispersive decay for solutions to the mass-critical NLS  with initial data $u_0\in L^2(\mathbb{R}^d)\cap  L^{r'}(\mathbb{R}^d)$ for $d=1,2,3$. Their work is optimal in the sense that no auxiliary assumptions are made besides finiteness of the critical norm, which is essential for the existence of solutions. Pointwise decay estimate for the energy-critical nonlinear wave equation was obtained by Looi \cite{Loo24}.

In this paper, we discuss the decay behavior of solutions to the mass-critical generalized KdV equation. To be precise, we show analogues of \eqref{jcKdV-DisDecay} for nonlinear solutions.

\begin{theorem} \label{MainResult1}
Assume that $u(t)$ is the global strong solution to the defocusing mass-critical generalized KdV \eqref{gKdV} with initial data $u_0\in H^{\frac{1}{2}}\cap L^{1}$ and $\|u_0\|_{H^{\frac{1}{2}}}\ll1$, then there exists a constant $C=C\big(\|u_0\|_{ H^{\frac{1}{2}}}\big)$, such that 
	\begin{align}
\|u(t,x)\|_{L^{\infty}_{x}}\leq C |t|^{-\frac{1}{3}} \|u_0\|_{L^{1}_{x}}.   \label{MainResult1a}
	\end{align}
\end{theorem}

In fact, the long time asymptotic behavior of solution to the gKdV equation
\begin{equation}
	\left\{
	\begin{aligned}
		&\partial_{t}u +\partial_{x}^{3} u + u^k\partial_{x} u = 0, \\
		&u(0,x)=u_0(x), \label{gKdV2} \\
	\end{aligned}
	\right.
\end{equation}
has been intensively studied by many authors.

Kenig, Ponce and Vega \cite{KPV93}  obtained the small data global well-posedness in critical spaces $\dot{H}^{s_k}$ with $s_k=1/2-k/2$ for the IVP \eqref{gKdV2} when $k>4$. 
\begin{theorem}[Kenig-Ponce-Vega \cite{KPV93}, 1993] \label{KPV-gKdV-Global-kg4}
Let $k>4$ and $s_k=1/2-k/2$. Then there exists $\delta_k>0$ such that for any $u_0 \in \dot{H}^{s_k}(\mathbb{R})$  with
 $$\|D_x^{s_k}u_0\|_{ L^2(\mathbb{R})}<\delta_k,$$
there exists a unique strong solution $u(t)$ of the IVP \eqref{gKdV2} satisfying
	\begin{align}
	&  u\in C\big(\mathbb{R}; \dot{H}^{s_k}(\mathbb{R}) \big)\cap L^{\infty}\big(\mathbb{R}; \dot{H}^{s_k}(\mathbb{R}) \big),\label{KPV-gKdV-Global-kg4-A1}  \\
	&   \left\|D_x^{s_k} u_x \right\|_{L^{\infty}_{x}L^{2}_{t}}< \infty,
\hspace{5mm} 	\|D_x^{s_k}u\|_{L^{5}_{x}L^{10}_{t}}< \infty,\label{KPV-gKdV-Global-kg4-A2}
	\end{align}
and 
	\begin{align}
		\big\|D_x^{\frac{1}{10}-\frac{2}{5k}}D_t^{\frac{3}{10}-\frac{6}{5k}}u\big\|_{L^{p_k}_{x}L^{q_k}_{t}}< \infty. \label{KPV-gKdV-Global-kg4-A3}
	\end{align}
where $\frac{1}{p_k}=\frac{2}{5k}+\frac{1}{10}$, $\frac{1}{q_k}=\frac{3}{10}-\frac{4}{5k}$. 

Moreover, the map $u_0 \to u(t)$ from  $\{u_0 \in \dot{H}^{s_k}(\mathbb{R}) | \|D_x^{s_k}u_0\|_{ L^2(\mathbb{R})}<\delta_k\}$ into
the class defined by \eqref{KPV-gKdV-Global-kg4-A1}--\eqref{KPV-gKdV-Global-kg4-A3} is Lipschitz.
\end{theorem}

 Strauss \cite{Stra74}  proved that solutions to \eqref{gKdV2} with $k>4$ decay with the same speed as solutions to the corresponding linear equation, i.e.,
$$\sup_{x\in \mathbb{R}}|u(t, x)|\leq C(1+t)^{-1/3}, \hspace{5mm} \text{as} \hspace{2mm} t\to \infty,$$
if $u_0\in H^{1}\cap L^{1}$ and $\|u_0\|_{L^{1}}+\|u_0\|_{L^{2}}\ll1$. Later, this result was extended to $k>(19-\sqrt{57})/4\approx 2.86$ in a range of articles \cite{ Klai82, KlaiPon83, Stra81, Shatah82, PV90}.  By assuming that the initial data $u_0$ of \eqref{gKdV2} lie  in the weighted Sobolev space $H^{1,1}:=\{f \in L^2; \ \|(1+|x|^2)^{1/2}(1-\partial^2_x)^{1/2}f\|<\infty\}$ and $\|u_0\|_{ H^{1,1}}\ll 1$, for $k>2$ Hayashi and Naumkin  \cite{HayNa98}  established the decay estimate
$$\|u(t)\|_{L^{r}}\leq C(1+t)^{-1/3(1-1/r)}, \hspace{5mm} r\in (4, \infty].$$

Using the perturbation theory and explicit representation of Fourier transform of the nonlinearity,  Naumkin and Shishmarev \cite{NaShi96} obtained the asymptotic expansion for solutions to \eqref{gKdV2} with the integer power of nonlinearity not less than 3 ($k\geq 3$). It is worth mentioning that Rammaha \cite{Rammaha89}  showed solutions to \eqref{gKdV2} with $k\in [0, 1]$ were not asymptotically free and he proposed a conjecture that it will not be asymptotically free also for the case $k\in [1, 2]$. 

Ifrim, Koch and Tataru \cite{IKT2023} investigated dispersive decay for solutions to the KdV equation with small localized initial data. They proved that if the initial data $u_0$ satisfies
$$\|u_0\|_{\dot{B}^{-\frac{1}{2}}_{2, \infty}}+\|xu_0\|_{\dot{H}^{\frac{1}{2}}}\leq \varepsilon \ll 1,$$
then the linear dispersive decay persists for the nonlinear problem on time scale $T_{\varepsilon}=\varepsilon^{-3}$. In other words, the  time scale that marks the earliest possible emergence of either solitons or dispersive shocks is $\varepsilon^{-3}$. 

The second result of this paper is the dispersive estimate for solutions to gKdV equations \eqref{gKdV2} with $k>4$. 

\begin{theorem} \label{MainResult2}
Let $k\in \mathbb{N}$ and  $k>4$. Assume that $u(t)$ is the global strong solution to  \eqref{gKdV2} with initial data $u_0\in H^{\frac{1}{2}}\cap L^{1}$ and $\|u_0\|_{H^{\frac{1}{2}}}\ll1$, then there exists a constant $C=C\big(\|u_0\|_{ H^{\frac{1}{2}}}\big)$, such that 
	\begin{align}
\|u(t,x)\|_{L^{\infty}_{x}(\mathbb{R})}\leq C |t|^{-\frac{1}{3}} \|u_0\|_{L^{1}_{x}(\mathbb{R})}.   \label{MainResult2a}
	\end{align}
\end{theorem}

\begin{remark} 
Comparing the previous dispersive decay results for gKdV equations and Theorem \ref{MainResult1} and Theorem \ref{MainResult2} here, we see that the regularity of initial data is lowed from $H^{1}$ to $H^{1/2}$.
\end{remark} 

Besides, we establish dispersive estimates for solutions to the gZK equation \eqref{gZK} which may be seen as a natural multi-dimensional extension of the gKdV equation \eqref{gKdV2}. Our three theorems treat spatial dimensions two, three and four respectively.
\begin{theorem} \label{MainResult3}
Let $4< r \leq \infty$, $\frac{1}{r}+\frac{1}{r'}=1$, $k\in \mathbb{N}$ and  $k\geq 3$. Assume that $u(t)$ is the global strong solution to the 2D gZK equation \eqref{gZK} with small initial data $u_0$  satisfying
\begin{equation}
	\left\{
	\begin{aligned}
		&u_0\in H^{1}\cap L^{r'}, \hspace{3mm}  \|u_0\|_{H^{1}}\ll1,    \hspace{5mm}  \text{if} \hspace{3mm}  k=3,\\
		&u_0\in H^{2}\cap L^{r'}, \hspace{3mm} \|u_0\|_{H^{2}}\ll1,    \hspace{5mm}   \text{if} \hspace{3mm}  k\geq 4, \nonumber \\ 
	\end{aligned}
	\right.
\end{equation}
 then there exists a constant $C$ dependent on $u_0$, such that 
\begin{align}
\| u(t,x,y)  \|_{L^{r}_{xy}(\mathbb{R}^2)} \leq C |t|^{-\frac{2}{3}(1-\frac{2}{r})}\|u_0\|_{L^{r'}_{xy}(\mathbb{R}^2)}. \label{MainResult3a}
	\end{align}
\end{theorem}

\begin{theorem} \label{MainResult4}
Let $4< r \leq \infty$, $\frac{1}{r}+\frac{1}{r'}=1$, $d=3$ and  $k=4$. Assume that $u(t)$ is the global strong solution to the 3D energy-critical gZK equation \eqref{gZK} with small initial data 
$u_0\in H^{2}\cap L^{r'}$ and  $\|u_0\|_{H^{2}}\ll1$,
 then there exists a constant $C=C(\|u_0\|_{H^{2}})$, such that 
\begin{align}
\| u(t,x,\mathbf{y})  \|_{L^{r}_{x\mathbf{y}}(\mathbb{R}^3)} \leq C(\|u_0\|_{H^{2}}) |t|^{-(1-\frac{2}{r})}\|u_0\|_{L^{r'}_{x\mathbf{y}}(\mathbb{R}^3)}. \label{MainResult4a}
	\end{align}
\end{theorem}

\begin{theorem} \label{MainResult5}
Let $d=4$ and  $k=3$. Assume that $u(t)$ is the global strong solution to the 4D energy-subcritical gZK equation \eqref{gZK} with small initial data 
$u_0\in H^{2}_{x\mathbf{y}}\cap (-\Delta)^{-\frac{1}{2}}L^{1}_{x\mathbf{y}}\cap L^{\frac{6}{5}}_{\mathbf{y}}L^2_{x}$ and  $\|u_0\|_{H^{2}_{x\mathbf{y}}(\mathbb{R}^4)}\ll1$. Then there exists a constant $C=C(\|u_0\|_{H^{2}_{x\mathbf{y}}})$  such that 
\begin{align}
\| u(t,x,\mathbf{y})  \|_{L^{\infty}_{x\mathbf{y}}(\mathbb{R}^4)} \leq C(\|u_0\|_{H^{2}_{x\mathbf{y}}}) |t|^{-1}\left(1+\|(-\Delta)^{\frac{1}{2}}u_0\|_{L^{1}_{x\mathbf{y}}(\mathbb{R}^4)}\right),\label{MainResult5AA}
	\end{align}
and
\begin{align}
\| \partial_x u(t,x,\mathbf{y}) \|_{L^{6}_{\mathbf{y}}L^2_{x}(\mathbb{R}^4)} \lesssim |t|^{-1}. \label{MainResult5BB}
	\end{align}
\end{theorem}

{\bf{Notations.}}  We say $A\lesssim B$ if $A\leq C \cdot B$ for an absolute constant $C>0$. $A\ll B$ means that  $A<C \cdot B$ for a very large positive constant $C$. Our conventions for the Fourier transform are
\begin{equation}
\widehat{f}(\xi)=\frac{1}{\sqrt{2\pi}}\int f(x)e^{-ix\xi}dx,\hspace{5mm}
\big(\mathscr F^{-1}f\big)(x)=\frac{1}{\sqrt{2\pi}}\int f(x)e^{ix\xi}dx. \nonumber
\end{equation}

Let $N\in 2^{\mathbb{N}}$  and $\chi: \mathbb{R} \to \mathbb{R}$ be a Schwartz function supported in  $[\frac{1}{2}, 2]$. Set $\chi_1(\cdot)=1-\sum_{N\in 2^{\mathbb{N}}}\chi(N^{-1}\cdot)$.  Define the  inhomogeneous Littlewood-Paley projection operator $P_N$ by 
$$P_1f:=\mathscr{F}^{-1}\left(\chi_1(|\xi|)  \widehat{f}(\xi) \right) , \hspace{5mm} P_Nf:=\mathscr{F}^{-1}\left(\chi(N^{-1}|\xi|)  \widehat{f}(\xi) \right). $$

For $s>0$, we define $D^sf$ and $J^sf$ as
$$\widehat{D^s}f(\xi)=|\xi|^{s}\widehat{f}(\xi)\hspace{4mm} and \hspace{4mm} \widehat{J^s}f(\xi)=(1+|\xi|^2)^{s/2}\widehat{f}(\xi).$$

{\bf{Organization of the paper.}} The rest of the article is organized as follows. In Section 2, we include some basic estimates that will be used frequently later. In Section 3, we show dispersive estimates for solutions to the mass-critical gKdV and  gKdV equation with $k\geq 6$. The proof of Theorem \ref{MainResult1} is more complex than that of Theorem \ref{MainResult2}. With the help of Lorents-Strichartz estimates, we show dispersive estimates for solutions to two dimensional gZK equations in Section 4. In Section 5, we give the proof of Theorem \ref{MainResult4} and Theorem \ref{MainResult5}.

\section{Preliminaries}\label{section:linear estimate}

In this section, we  recall Strichartz estimates, Kato's local smoothing estimates and the fractional Leibniz rules which will be useful to handle the nonlinear term with derivative.
\subsection{Strichartz estimates and Kato smoothing estimates for KdV} 
The standard dispersive estimates and Strichartz estimates for the free KdV operator $e^{-t\partial^3_x}$ read as follows. 
\begin{lemma}[Dispersive estimates, \cite{KPV89}] \label{PrelDisDecay}
We have
	\begin{align}
		\left\|D_x^{\theta \alpha}e^{-t\partial^3_x}u_0\right\|_{L^r} \lesssim  |t|^{-\frac{\alpha+1}{3}\theta } \left\|u_0\right\|_{L^{r'}} ,\label{PrelDisDecay1}
	\end{align}
where $\theta=\frac{r-2}{r}$, $2\leq r \leq \infty$ and $0\leq \alpha \leq 1/2$. In particular, by taking $\alpha=1/2$, we deduce
	\begin{align}
		\left\|D_x^{\frac{1}{2}-\frac{1}{r}}e^{-t\partial^3_x}u_0\right\|_{L^r} \lesssim  |t|^{-(\frac{1}{2}-\frac{1}{r})} \left\|u_0\right\|_{L^{r'}} ,\label{PrelDisDecay2}
	\end{align}
\end{lemma}

\begin{lemma}[Strichartz estimates, \cite{KPV89}]\label{PrelStrichartz}
	Let $\theta\in [0, 1]$, $ \alpha \in [0, 1/2]$, $(q,p)=(\frac{6}{\theta(\alpha+1)}, \frac{2}{1-\theta})$  and $\frac{1}{p}+\frac{1}{p'}=\frac{1}{q}+\frac{1}{q'}=1$. Then
	\begin{align}
		\left\|D_x^{\frac{\theta\alpha}{2}}e^{-t\partial^3_x}u_0\right\|_{L_t^qL^p_{x}} &\lesssim
		\|u_0\|_{L^2_{x}}, \label{Stri1} \\
	\left\|\int_{-\infty}^{\infty} D_x^{\theta\alpha}e^{-(t-s)\partial^3_x}g(s, \cdot)ds\right\|_{L_t^qL^p_{x}} &\lesssim
		\|g\|_{L_t^{q'}L^{p'}_{x}}. \label{Stri2}
	\end{align}
\end{lemma}

Kato smoothing effect can help to absorb derivative arising from the nonlinear term.
\begin{lemma}
	[Local smoothing estimates, \cite{KPV93}]\label{KatoSmootEff} We have
\begin{align}
\big\|\partial_x e^{-t\partial^3_x}u_0 \big\|_{L^2_{t}} &=c
		\|u_0 \|_{L^{2}_{x}},\label{KatoSmootEff1} \\
	\left\|\partial_x\int_0^t e^{s\partial^3_x}g( s, \cdot)\mathrm{d}s\right\|_{L_x^{2}} &\lesssim
		\|g\|_{L^1_xL^{2}_{t}},\label{KatoSmootEff2}\\
	\left\|\partial_x^{2}\int_0^t e^{-(t-s)\partial^3_x}g( s, \cdot)\mathrm{d}s\right\|_{L_x^{\infty}L^2_{t}} &\lesssim
		\|g\|_{L^1_xL^{2}_{t}}.\label{KatoSmootEff3}
	\end{align}
\end{lemma}

\begin{lemma}[Maximal function estimate, \cite{KPV93}]\label{gKdVMaxFun}
We have
\begin{align}
		\big\|e^{-t\partial^3_x} u_0\big\|_{L_{x}^{4} L_{t}^{\infty}} \lesssim  \|u_0\|_{\dot{H}^{\frac{1}{4}}_x}. \label{gKdVMaxFun1}
	\end{align}
\end{lemma}
\begin{lemma} \label{gKdVL(XT)Norm}
Let $\theta\in [0, 1]$ and $(p,q)=(\frac{4}{\theta}, \frac{2}{1-\theta})$. Then, we have
	\begin{equation}
		\big\|D_x^{1-\frac{5\theta}{4}}e^{-t\partial^3_x} u_0\big\|_{L_{x}^{p} L_{t}^{q}} \lesssim  \|u_0\|_{L^{2}_x},  \label{gKdVL(XT)Norm1}
	\end{equation}
and 
	\begin{equation}
		\left\|D_x^{2-\frac{5}{p_1}-\frac{5}{p_2}}\int_{0}^{t} e^{-(t-s)\partial^3_x}g(s) ds\right\|_{L_{x}^{p_1} L_{t}^{q_1}} \lesssim  \|g\|_{L_{x}^{p_2'} L_{t}^{q_2'}},  \label{gKdVL(XT)Norm1a}
	\end{equation}
where $\frac{4}{p_1}+\frac{2}{q_1}=\frac{4}{p_2}+\frac{2}{q_2}=1$, $\frac{1}{p_2}+\frac{1}{p'_2}=\frac{1}{q_2}+\frac{1}{q'_2}=1$ and $2\leq q_1, q_2  \leq \infty$.

 Taking $\theta=\frac{2}{5}$  in \eqref{gKdVL(XT)Norm1} yields that 
	\begin{equation}
		\big\|D_x^{\frac{1}{2}}e^{-t\partial^3_x} u_0\big\|_{L_{x}^{10} L_{t}^{\frac{10}{3}}} \lesssim  \|u_0\|_{L^{2}_x}.  \label{gKdVL(XT)Norm2}
	\end{equation}
In particular, if $u(t)$ is the solution to \eqref{gKdV} with initial data $u_0$ given in Theorem  \ref{KPV-gKdV-Globalscattering}, then
	\begin{equation}
		\big\|D_x^{\frac{1}{2}}u\big\|_{L_{x}^{10} L_{t}^{\frac{10}{3}}} <  \infty.  \label{gKdVL(XT)Norm2b}
	\end{equation}

More generally, if $u(t)$ is the global solution to \eqref{gKdV} with  small initial data $\|u_0\|_{H^{s}}$, then 
	\begin{equation}
 \big\|D_x^{1-\frac{5\theta}{4}+s}u\big\|_{L_{x}^{p} L_{t}^{q}}<  \infty.  \label{gKdVL(XT)Norm3}
	\end{equation}
\end{lemma}

\begin{proof}
	\eqref{gKdVL(XT)Norm1} can be derived by  interpolating \eqref{KatoSmootEff1} and \eqref{gKdVMaxFun1}. This method is form Kenig, Ponce and Vega, see Corollary 3.8 in \cite{KPV93}. 	\eqref{gKdVL(XT)Norm1a} is dual to	\eqref{gKdVL(XT)Norm1}.
	
If $u(t)$ is the solution to \eqref{gKdV}, then
\begin{align}
u(t)=e^{-t\partial^3_x}u_0+\int_{0}^{t} e^{-(t-s)\partial^3_x}\partial_x u^5(s) ds \nonumber
	\end{align}
which yields by using  \eqref{gKdVL(XT)Norm2}, \eqref{gKdVL(XT)Norm1a} with $(p_1, q_1)=(10, \frac{10}{3})$ and $(p_2, q_2)=(\infty, 2)$ and  \eqref{KPV-gKdV-GlobalscatteringA2}  that
\begin{align}
\big\|D_x^{\frac{1}{2}}u\big\|_{L_{x}^{10} L_{t}^{\frac{10}{3}}}&\lesssim \big\|D_x^{\frac{1}{2}}e^{-t\partial^3_x}u_0\big\|_{L_{x}^{10} L_{t}^{\frac{10}{3}}}+\left\|D_x^{\frac{1}{2}}\partial_x\int_{0}^{t} e^{-(t-s)\partial^3_x} u^5(s) ds \right\|_{L_{x}^{10} L_{t}^{\frac{10}{3}}}\notag \\
&\lesssim \|u_0\|_{L_{x}^{2}}+\left\|u^5 \right\|_{L_{x}^{1} L_{t}^{2}}
\lesssim \|u_0\|_{L_{x}^{2}}+\left\|u \right\|^5_{L_{x}^{5} L_{t}^{10}}<\infty.\nonumber
	\end{align}

\eqref{gKdVL(XT)Norm3} can be obtained using a similar argument. Thus, we finish the proof.
\end{proof}
\begin{lemma} \label{gKdVLNorm L245xLqqt}
Assume that  $u(t)$ is the global solution to the mass-critical generalized KdV equation \eqref{gKdV} with  small initial data $\|u_0\|_{H^{\frac{1}{3}}}\ll 1$, then 
	\begin{equation}
 \|u\|_{L_{x}^{\frac{24}{5}} L_{t}^{\infty}}<  \infty.   \label{gKdVLNorm L245xLqqt0}
	\end{equation}
\end{lemma}

\begin{proof} According to Duhamel formula, we have
\begin{align}
u(t)=e^{-t\partial^3_x}u_0+\int_{0}^{t} e^{-(t-s)\partial^3_x}\partial_x u^5(s) ds . \nonumber
	\end{align}
Hence,
\begin{align}
 \|u\|_{L_{x}^{\frac{24}{5}} L_{t}^{\infty}}\leq \big\|e^{-t\partial^3_x}u_0\big\|_{L_{x}^{\frac{24}{5}} L_{t}^{\infty}}+\left\|\int_{0}^{t} e^{-(t-s)\partial^3_x}\partial_x u^5(s) ds \right\|_{L_{x}^{\frac{24}{5}} L_{t}^{\infty}}:=T_1+T_2. \nonumber
	\end{align}

We estimate the first part by using Sobolev inequality
\begin{align}
T_1= \big\|e^{-t\partial^3_x}u_0\big\|_{L_{x}^{\frac{24}{5}} L_{t}^{\infty}}&\lesssim \big\|e^{-t\partial^3_x}u_0\big\|_{L_{x}^{\frac{24}{5}} L_{t}^{12}}+\big\|D^{1/12}_te^{-t\partial^3_x}u_0\big\|_{L_{x}^{\frac{24}{5}} L_{t}^{12}}\notag \\
&\lesssim \|u_0\|_{L_{x}^{2}}+\big\|D^{1/12}_te^{-t\partial^3_x}u_0\big\|_{L_{x}^{\frac{24}{5}} L_{t}^{12}}. \label{gKdVLNorm L245xLqqt1}
	\end{align}
Observe that
\begin{align}
D^{1/12}_te^{-t\partial^3_x}u_0&=\int_{\mathbb{R}} e^{ix\xi} \widehat{u_0}(\xi) d\xi \int_{\mathbb{R}} e^{it\tau}(i\tau)^{1/12}\delta(\tau-\xi^3)  d\tau  \notag \\
&=\int_{\mathbb{R}} e^{ix\xi} (i\xi^3)^{1/12} e^{it\xi^3}\widehat{u_0}(\xi) d\xi=e^{i\frac{\pi}{12}}D^{1/4}_xe^{-t\partial^3_x}u_0, \nonumber
	\end{align}
so by taking use of \eqref{gKdVL(XT)Norm1} with $\theta=\frac{5}{6}$ one gets
\begin{align}
T_1\lesssim \|u_0\|_{L_{x}^{2}}+\big\|D^{1/4}_xe^{-t\partial^3_x}u_0\big\|_{L_{x}^{\frac{24}{5}}L_{t}^{12}}\lesssim \|u_0\|_{H_{x}^{\frac{7}{24}}}. \label{gKdVLNorm L245xLqqt2}
	\end{align}

Next, we want to handle the second part. Applying the same strategy as Proposition 3.13 in \cite{KPV93} introduced by Kenig, Ponce and Vega, one can get 
\begin{align}
\left\|D^{\frac{1}{12}}_t\int_{0}^{t} e^{-(t-s)\partial^3_x}g(s) ds \right\|_{L_{x}^{\frac{24}{5}} L_{t}^{12}}\lesssim \big\|g \big\|_{L_{x}^{\frac{15}{13}} L_{t}^{\frac{30}{23}}}. \label{gKdVLNorm L245xLqqt3}
	\end{align}
Then, \eqref{gKdVL(XT)Norm1a} and \eqref{gKdVLNorm L245xLqqt3} help imply
\begin{align}
T_2\lesssim& \|u^4u_x\|_{L_{x}^{\frac{120}{97}} L_{t}^{\frac{60}{53}}}+\|u^4u_x \|_{L_{x}^{\frac{15}{13}} L_{t}^{\frac{30}{23}}} \notag \\
\lesssim& \|u\|^4_{L_{x}^{5} L_{t}^{10}}\|u_x\|_{L_{x}^{120} L_{t}^{\frac{60}{29}}}+\|u\|^4_{L_{x}^{5} L_{t}^{10}}\|u_x\|_{L_{x}^{15} L_{t}^{\frac{30}{11}}}\lesssim C(\|u_0\|_{H^{\frac{1}{3}}}). \label{gKdVLNorm L245xLqqt4}
	\end{align}
This finishes the proof.
\end{proof}

\subsection{Strichartz estimates for ZK} 
 
Let us recall the dispersive estimate for 2D gZK equation. We denote $U(t):=e^{-t\partial_x\Delta}$ for simplicity.
 
\begin{lemma}[see Lemma 2.3 in \cite{LP09}] \label{2dZK-PrelDisDecay}
Let $0\leq \alpha \leq \frac{1}{2}$ and $0\leq \theta\leq 1$. Then,
	\begin{align}
		\left\|D_x^{\theta \alpha}U(t)u_0\right\|_{L^r_{xy}} \lesssim  |t|^{-\frac{\alpha+2}{3}\theta } \left\|u_0\right\|_{L^{r'}_{xy}} ,\label{2dZK-PrelDisDecay1}
	\end{align}
where $r=\frac{2}{1-\theta}$  and $\frac{1}{r}+\frac{1}{r'}=1$. 
\end{lemma}

We will employ refinements of  Strichartz estimates in Lorentz spaces \cite{Grafa14}. It was Fan, Killip, Visan and Zhao \cite{FKVZ24} who first used Lorentz spaces to study dispersive decay for solutions to the mass-critical nonlinear Schr\"odinger equation.

\begin{definition}[Lorentz spaces] Let $1\leq p<\infty$ and $1\leq q\leq\infty$. The Lorentz space $L^{p,q}(\mathbb{R}^d)$ is the space of measurable functions $f:\mathbb{R}^d \to \mathbb{C}$ for which the quasinorm
$$\|f\|_{L^{p,q}(\mathbb{R}^d)}=p^{\frac{1}{q}}\left\| \lambda \big| \{x\in \mathbb{R}^d: |f(x)|>\lambda \}\big|^{\frac{1}{p}} \right\|_{L^{q}((0, \infty), \frac{d\lambda}{\lambda})}$$
is finite. Here, $|A|$ denotes the Lebesgue measure of the set $A\subseteq \mathbb{R}^d$.
\end{definition}

Note that $L^{p,p}(\mathbb{R}^d)\simeq L^{p}(\mathbb{R}^d)$, and $L^{p,\infty}(\mathbb{R}^d)$ coincides with the weak $L^{p}(\mathbb{R}^d)$ space. If $1\leq p<\infty$ and $1\leq q<r\leq\infty$, then  $L^{p,q}(\mathbb{R}^d)\hookrightarrow  L^{p,r}(\mathbb{R}^d)$. 

\begin{lemma}[H\"older inequality in Lorentz spaces] Let $1\leq p,p_1,p_2<\infty$, $1\leq q,q_1,q_2\leq\infty$, and $\frac{1}{p}=\frac{1}{p_1}+\frac{1}{p_2}$, $\frac{1}{q}=\frac{1}{q_1}+\frac{1}{q_2}$. Then,
$$\|fg\|_{L^{p,q}(\mathbb{R}^d)}\lesssim \|f\|_{L^{p,q}(\mathbb{R}^d)}\|g\|_{L^{p,q}(\mathbb{R}^d)}.$$
\end{lemma}

Applying Hardy-Littlewood-Sobolev inequality in Lorentz spaces yields the following Lorentz-space improvement:
\begin{lemma}[Lorentz-Strichartz estimates]\label{2dZK-PrelStrichartz}
	Let $2\leq p,q\leq \infty$, $\frac{3}{q}+\frac{2}{p}=1$  and $\frac{1}{p}+\frac{1}{p'}=\frac{1}{q}+\frac{1}{q'}=1$. Then
	\begin{align}
		\left\|U(t)u_0\right\|_{L_t^{q,2}L^p_{xy}(\mathbb{R}\times\mathbb{R}^2)} &\lesssim
		\|u_0\|_{L^2_{xy}(\mathbb{R}^2)}, \label{2dZK-Stri1} \\
	\left\| \int_{0}^{t} U(t-s)g(s, \cdot)ds\right\|_{L_t^{q,2}L^p_{xy}(\mathbb{R}\times\mathbb{R}^2)} &\lesssim
		\|g\|_{L_t^{q',2}L^{p'}_{xy}(\mathbb{R}\times\mathbb{R}^2)}. \label{2dZK-Stri2}
	\end{align}
\end{lemma}

\subsection{Kato-Ponce commutator estimates} 

Kato-Ponce commutator estimates \cite{KP1988} play an important role in the well-posedness theory for the KdV equation.

\begin{lemma}[Kato-Ponce 1988, \cite{KP1988}]\label{KPlem:Leibniz}
	Let $s>0$ and $p\in (1,\infty)$. Then
	\begin{equation}
		\|J^s(f g)-f J^s g\|_{L^p}\lesssim \|g\|_{L^\infty}\|J^sf\|_{L^p}+\|\partial f\|_{L^\infty}\|J^{s-1}g\|_{L^p}. \label{KPlem:Leibniz-1}
	\end{equation}
\end{lemma}

\begin{lemma}[see Theorem 1 in \cite{KPV93}]\label{lem:Leibniz}
	Let $s\in(0,1)$ and $p\in (1,\infty)$. Then
	\begin{equation}
		\|D^s(f g)-f D^s g-g D^s f\|_{L^p(\mathbb{R})}\lesssim \|g\|_{L^\infty(\mathbb{R})}\|D^sf\|_{L^p(\mathbb{R})}. \label{lem:Leibniz-1}
	\end{equation}
Further more, 
	\begin{equation}
		\|D^s(f g)\|_{L^p(\mathbb{R})}\lesssim \|f D^s g\|_{L^p(\mathbb{R})} + \|g\|_{L^\infty(\mathbb{R})}\|D^s f\|_{L^p(\mathbb{R})}. \label{lem:Leibniz-12}
	\end{equation}
\end{lemma}

We need the following fractional Leibniz rule that includes the end-point situation.

\begin{lemma}[see Theorem 1.2 in \cite{Li19}]\label{lem:Leibnizpoly}
	Let $s>0$, $1\leq p<\infty$ and  $1< p_1,  p_2<\infty$ with $1/p=1/p_1+1/p_2$. Then for any $s_1, s_2 \geq 0$ with $s_1+ s_2=s$, and any
$f, g \in \mathcal{S}(\mathbb{R}^d)$, the following inequality holds:
	\begin{equation}
		\Big\|D^s(f g)-\sum_{|\alpha|\leq s_1}\frac{1}{\alpha!}\partial_x^{\alpha}f D^{s,\alpha} g-\sum_{|\beta|\leq s_2}\frac{1}{\beta!}\partial_x^{\beta}g D^{s,\beta} f\Big\|_{L^p(\mathbb{R}^d)}\lesssim \|D^{s_1}f\|_{L^{p_1}(\mathbb{R}^d)} \|D^{s_2}g\|_{L^{p_2}(\mathbb{R}^d)}  \label{lem:Leibnizpoly1}
	\end{equation}
where the operator $D^{s,\alpha}$ is defined via Fourier transform  as $$\widehat{D^{s,\alpha}g}(\xi)=i^{-|\alpha|}\partial_{\xi}^{\alpha}|\xi|^s .$$ In particular, for $0<s<1$, by taking $s_1=0$ and $p=1$, one has
	\begin{equation}
		\big\|D^s(f g)\big\|_{L^1(\mathbb{R}^d)}\lesssim \big\|g D^{s} f\big\|_{L^1(\mathbb{R}^d)}+\|f\|_{L^{2}(\mathbb{R}^d)} \|D^{s}g\|_{L^{2}(\mathbb{R}^d)}.  \label{lem:Leibnizpoly2}
	\end{equation}
\end{lemma}

 \section{Dispersive decay for the mass-critical gKdV}\label{gKdVDispDecay}
This section is  devoted to show Theorem \ref{MainResult1} and Theorem \ref{MainResult2}.

{\bf{Proof of Theorem \ref{MainResult1}.}} By time-reversal symmetry, we only need to show the pointwise dispersive estimate \eqref{MainResult1a}  for $t>0$. For $T\in (0, \infty]$, denote
$$\|u \|_{X(T)}:=\sup_{t\in [0, T]}|t|^{\frac{1}{3}}\| u(t)\|_{L^{\infty}_{x}}+\sup_{t\in [0, T]}|t|^{\frac{1}{6}}\|D_x^{1/2} u(t)\|_{L^{\infty}_{x}}.$$

\noindent
$\bullet$ {\bf Estimate for  $\| u(t)\|_{L^{\infty}_{x}}$.}  
\vspace{1mm}

By Duhamel's principle, the solution to \eqref{gKdV} can be written as
\begin{align}
u(t)&=e^{-t\partial^3_x}u_0+\int_{0}^{t} e^{-(t-s)\partial^3_x}\partial_x u^5(s) ds \notag \\
&=e^{-t\partial^3_x}u_0+\int_{0}^{\frac{t}{2}} e^{-(t-s)\partial^3_x}\partial_x u^5(s) ds+\int_{\frac{t}{2}}^{t} e^{-(t-s)\partial^3_x}\partial_x u^5(s) ds. \label{gKdVsol}
	\end{align}
Using dispersive decay estimate \eqref{PrelDisDecay1} with $\alpha=0$ and $r=\infty$,  we see that the first term on RHS\eqref{gKdVsol}
can be controlled by:
\begin{align} \|e^{-t\partial^3_x}u_0\|_{L^{\infty}_x}\lesssim t^{-\frac{1}{3}}\|u_0\|_{L^{1}_x}. \label{DispDecayLinear}
\end{align}

Next, we estimate the last two term on RHS\eqref{gKdVsol} respectively. Applying  \eqref{PrelDisDecay2} with  $r=\infty$ and commutator estimate \eqref{lem:Leibnizpoly2},  we deduce that
\begin{align}
&\left\| \int_{0}^{\frac{t}{2}}e^{-(t-s)\partial^3_x}\partial_x u^5(s) ds \right\|_{L^{\infty}_{x}} \notag \\
\lesssim & \int_{0}^{\frac{t}{2}}|t-s|^{-\frac{1}{2}}\big\|D^{\frac{1}{2}}_x u^5(s) \big\|_{L^{1}_{x}}ds \notag \\
\lesssim &  \int_{0}^{\frac{t}{2}}|t-s|^{-\frac{1}{2}}\big\|u^4D^{\frac{1}{2}}_x u(s) \big\|_{L^{1}_{x}}ds+\int_{0}^{\frac{t}{2}}|t-s|^{-\frac{1}{2}}\|u\|_{L^{2}_{x}}\big\|D^{\frac{1}{2}}_xu^4\big\|_{L^{2}_{x}}ds.   \label{DispDecay0Ia1}
	\end{align}

On one hand, it follows from  \eqref{KPV-gKdV-GlobalscatteringA1}, \eqref{KPV-gKdV-GlobalscatteringA2} and \eqref{gKdVL(XT)Norm2b} that
\begin{align}
\int_{0}^{\frac{t}{2}}|t-s|^{-\frac{1}{2}}\big\|u^4D^{\frac{1}{2}}_x u \big\|_{L^{1}_{x}}ds  
\lesssim&  |t|^{-\frac{1}{2}} \int_{0}^{\frac{t}{2}}\big\|u^2D^{\frac{1}{2}}_x u \big\|_{L^{2}_{x}}\|u \|_{L^{2}_{x}}\|u \|_{L^{\infty}_{x}}ds  \notag \\
\lesssim& |t|^{-\frac{1}{2}} \|u \|_{X(T)}\int_{0}^{\frac{t}{2}}s^{-\frac{1}{3}}\big\|u^2D^{\frac{1}{2}}_x u \big\|_{L^{2}_{x}}\|u \|_{L^{2}_{x}}ds  \notag \\
\lesssim& |t|^{-\frac{1}{2}}\|u \|_{X(T)} \|u \|_{L^{\infty}_{t}L^{2}_{x}}\big\|u^2D^{\frac{1}{2}}_x u \big\|_{L^{2}_{tx}} \left(\int_{0}^{\frac{t}{2}}s^{-\frac{2}{3}}ds\right)^{\frac{1}{2}} \notag \\
\lesssim&	|t|^{-\frac{1}{3}}\|u \|_{X(T)} \|u \|_{L^{\infty}_{t}L^{2}_{x}}\| u\|^2_{L_{x}^{5} L_{t}^{10}}\big\|D_x^{\frac{1}{2}} u\big\|_{L_{x}^{10} L_{t}^{\frac{10}{3}}} \notag \\
\leq&	C(\|u_0\|_{L^2})|t|^{-\frac{1}{3}}\|u \|_{X(T)}. \label{DispDecay0Ia2}
	\end{align}
On the other hand, using ‌‌\eqref{lem:Leibniz-12} we have
\begin{align}
 &\int_{0}^{\frac{t}{2}}|t-s|^{-\frac{1}{2}}\|u\|_{L^{2}_{x}}\big\|D^{\frac{1}{2}}_xu^4\big\|_{L^{2}_{x}}ds \notag \\
\lesssim &  \int_{0}^{\frac{t}{2}}|t-s|^{-\frac{1}{2}}\|u\|_{L^{2}_{x}}\big\|u^3D^{\frac{1}{2}}_x u \big\|_{L^{2}_{x}}ds+\int_{0}^{\frac{t}{2}}|t-s|^{-\frac{1}{2}}\|u\|_{L^{2}_{x}}\|u\|_{L^{\infty}_{x}}\big\|D^{\frac{1}{2}}_xu^3\big\|_{L^{2}_{x}}ds  \label{DispDecay0Ia3}
	\end{align}
of which the first term on the right-hand side can be controlled by the same way as \eqref{DispDecay0Ia2}. For the second term on RHS\eqref{DispDecay0Ia3}, applying Gagliardo-Nirenberg inequality 
$$\big\|D^{\frac{1}{2}}_xf\big\|_{L^{2}_{x}}\leq \|f\|^{\frac{1}{2}}_{L^{2}_{x}}\big\|D_x f\big\|^{\frac{1}{2}}_{L^{2}_{x}}\leq \|f\|_{L^{2}_{x}}+\big\|D_x f\big\|_{L^{2}_{x}}$$
yields
\begin{align}
 &\int_{0}^{\frac{t}{2}}|t-s|^{-\frac{1}{2}}\|u\|_{L^{2}_{x}}\|u\|_{L^{\infty}_{x}}\big\|D^{\frac{1}{2}}_xu^3\big\|_{L^{2}_{x}}ds  \notag \\
\lesssim &|t|^{-\frac{1}{2}}\int_{0}^{\frac{t}{2}}\|u\|_{L^{2}_{x}}\|u\|_{L^{\infty}_{x}}\big\|u^3\big\|_{L^{2}_{x}}ds+|t|^{-\frac{1}{2}}\int_{0}^{\frac{t}{2}}\|u\|_{L^{2}_{x}}\|u\|_{L^{\infty}_{x}}\big\|D_xu^3\big\|_{L^{2}_{x}}ds  \notag \\
\lesssim &|t|^{-\frac{1}{2}}\|u \|^2_{X(T)} \|u \|_{L^{\infty}_{t}L^{2}_{x}}\|u \|^{2}_{L^{12}_{t}L^{4}_{x}}\left(\int_{0}^{\frac{t}{2}}s^{-\frac{4}{5}}ds\right)^{\frac{5}{6}}\notag \\
&\hspace{3mm}+|t|^{-\frac{1}{2}}\|u \|_{X(T)} \|u \|_{L^{\infty}_{t}L^{2}_{x}}\big\|u^2u_x\big\|_{L^{2}_{tx}} \left(\int_{0}^{\frac{t}{2}}s^{-\frac{2}{3}}ds\right)^{\frac{1}{2}} \notag \\
\lesssim  &|t|^{-\frac{1}{3}}\|u \|^2_{X(T)} \|u \|_{L^{\infty}_{t}L^{2}_{x}}\|u\|^2_{L^{12}_{t}L^{4}_{x}} +|t|^{-\frac{1}{3}}\|u \|_{X(T)} \|u \|_{L^{\infty}_{t}L^{2}_{x}}\|u \|^2_{L^{4}_{x}L^{\infty}_{t}}\|u_x\|_{L^{\infty}_{x}L^{2}_{t}} \notag \\
\leq&C(\|u_0\|_{L^{2}})|t|^{-\frac{1}{3}}\|u \|^2_{X(T)}+	C(\|u_0\|_{H^{\frac{1}{4}}})|t|^{-\frac{1}{3}}\|u \|_{X(T)}. \label{DispDecay0Ia4}
	\end{align}
Collecting   \eqref{DispDecay0Ia1}, \eqref{DispDecay0Ia2}, \eqref{DispDecay0Ia3} and \eqref{DispDecay0Ia4} deduces that
\begin{align}
\left\| \int_{0}^{\frac{t}{2}}e^{-(t-s)\partial^3_x}\partial_x u^5(s) ds \right\|_{L^{\infty}_{x}} 
\leq C(\|u_0\|_{H^{\frac{1}{4}}})|t|^{-\frac{1}{3}}\|u \|_{X(T)}+C(\|u_0\|_{L^{2}})|t|^{-\frac{1}{3}}\|u \|^2_{X(T)}.   \label{DispDecay0Ia5}
	\end{align}

Let us consider the final term on RHS\eqref{gKdVsol}. Using dispersive decay estimate \eqref{PrelDisDecay1} and commutator estimate \eqref{lem:Leibnizpoly2}, we get
\begin{align}
&\left\| \int_{\frac{t}{2}}^{t} e^{-(t-s)\partial^3_x}\partial_x u^5(s) ds \right\|_{L^{\infty}_{x}} \notag \\
\lesssim& \int_{\frac{t}{2}}^{t}|t-s|^{-\frac{1}{2}}\big\|D^{\frac{1}{2}}_x u^5 \big\|_{L^{1}_{x}}ds \notag \\
\lesssim &  \int_{\frac{t}{2}}^{t}|t-s|^{-\frac{1}{2}}\big\|u^4D^{\frac{1}{2}}_x u \big\|_{L^{1}_{x}}ds +\int_{\frac{t}{2}}^{t}|t-s|^{-\frac{1}{2}}\|u\|_{L^{2}_{x}}\big\|D^{\frac{1}{2}}_xu^4\big\|_{L^{2}_{x}}ds. \label{DispDecay0IIb0}
	\end{align}
 For the first term on RHS\eqref{DispDecay0IIb0},
\begin{align}
  \int_{\frac{t}{2}}^{t}|t-s|^{-\frac{1}{2}}\big\|u^4D^{\frac{1}{2}}_x u \big\|_{L^{1}_{x}}ds 
\lesssim &\int_{\frac{t}{2}}^{t}|t-s|^{-\frac{1}{2}}\|u \|^2_{L^{\infty}_{x}}\big\|D^{\frac{1}{2}}_xu\big\|_{L^{\infty}_{x}}\|u\|^2_{L^{2}_{x}}ds\notag \\
\lesssim &t^{-\frac5 6}\|u \|^3_{X(T)}\|u\|^2_{L^{\infty}_{t}L^{2}_{x}}\int_{\frac{t}{2}}^{t}|t-s|^{-\frac{1}{2}}ds\notag \\
\leq& C(\|u_0 \|_{L^{2}_{x}})t^{-\frac1 3}\|u \|^3_{X(T)}.\label{DispDecay0IIb1}
	\end{align}
 For the second term on RHS\eqref{DispDecay0IIb0}, by using commutator estimate one can get
\begin{align}
  \int_{\frac{t}{2}}^{t}|t-s|^{-\frac{1}{2}}\big\|u^4D^{\frac{1}{2}}_x u \big\|_{L^{1}_{x}}ds 
\lesssim &\int_{\frac{t}{2}}^{t}|t-s|^{-\frac{1}{2}}\|u \|^3_{L^{\infty}_{x}}\|u\|_{L^{2}_{x}}\big\|D^{\frac{1}{2}}_xu\big\|_{L^{2}_{x}}ds\notag \\
\lesssim &t^{-1}\|u \|^3_{X(T)}\|u\|_{L^{\infty}_{t}L^{2}_{x}}\big\|D^{\frac{1}{2}}_xu\big\|_{L^{\infty}_{t}L^{2}_{x}}\int_{\frac{t}{2}}^{t}|t-s|^{-\frac{1}{2}}ds\notag \\
\leq& C(\|u_0 \|_{H^{\frac{1}{2}}_{x}})t^{-\frac1 2}\|u \|^3_{X(T)}.\label{DispDecay0IIb2}
	\end{align}
Then, inserting \eqref{DispDecay0IIb1} and \eqref{DispDecay0IIb2} into \eqref{DispDecay0IIb0} we obtain 
\begin{align}
\left\| \int_{\frac{t}{2}}^{t} e^{-(t-s)\partial^3_x}\partial_x u^5(s) ds \right\|_{L^{\infty}_{x}} \leq  C(\|u_0 \|_{H^{\frac{1}{2}}})t^{-\frac1 3}\|u \|^3_{X(T)}. \label{DispDecay0IIb3}
	\end{align}
which together with \eqref{DispDecay0Ia5} and \eqref{DispDecayLinear} derives 
\begin{align}
\sup_{t\in [0, T]}|t|^{\frac{1}{3}}\| u(t)  \|_{L^{\infty}_{x}} \leq  \|u_0 \|_{L^{1}_{x}}+C(\|u_0\|_{H^{\frac{1}{2}}})\left(\|u \|_{X(T)}+\|u \|^2_{X(T)}+\|u \|^3_{X(T)}\right). \label{DispDecay0mgKdVf1}
	\end{align}

\noindent
$\bullet$ {\bf Estimate for  $\|D_x^{1/2} u(t)\|_{L^{\infty}_{x}}$.}  
\vspace{1mm}

Note that
\begin{align}
D_x^{1/2}u(t)=D_x^{1/2}e^{-t\partial^3_x}u_0+D_x^{1/2}\int_{0}^{t} e^{-(t-s)\partial^3_x}\partial_x u^5(s) ds, \label{gKdVsol-deri}
	\end{align}
the linear component is easy to control. To estimate the contribution of the nonlinear term, we decompose the region of integration into $[0, \frac{t}{2}]$ and $[\frac{t}{2}, t]$.

By using  \eqref{PrelDisDecay2} with  $r=\infty$,  we have
\begin{align}
\left\| D^{\frac{1}{2}}_x\int_{0}^{\frac{t}{2}}e^{-(t-s)\partial^3_x}\partial_x u^5(s) ds \right\|_{L^{\infty}_{x}} 
\lesssim & \int_{0}^{\frac{t}{2}}|t-s|^{-\frac{1}{2}}\big\|u^4 u_x\big\|_{L^{1}_{x}}ds \notag \\
\lesssim &  |t|^{-\frac{1}{2}}\int_{0}^{\frac{t}{2}}\|u\|_{L^{2}_{x}}\|u\|_{L^{\infty}_{x}}\big\|u^2 u_x\big\|_{L^{2}_{x}}ds \notag \\
\lesssim &  |t|^{-\frac{1}{2}}\|u\|_{L^{\infty}_{t}L^{2}_{x}}\|u \|_{X(T)}\big\|u^2 u_x\big\|_{L^{2}_{tx}}\left(\int_{0}^{\frac{t}{2}}s^{-\frac{2}{3}}ds\right)^{\frac{1}{2}}\notag \\
\lesssim &  |t|^{-\frac{1}{3}}\|u\|_{L^{\infty}_{t}L^{2}_{x}}\|u \|_{X(T)}\|u\|^2_{L^{4}_{x}L^{\infty}_{t}}\big\|u_x\big\|_{L^{\infty}_{x}L^{2}_{t}} \notag \\
\leq& C(\|u_0 \|_{H^{\frac{1}{4}}})t^{-\frac1 3}\|u \|_{X(T)}.   \label{DispDecay0Ia1-deri}
	\end{align}
Furthermore,
\begin{align}
&\left\| D^{\frac{1}{2}}_x\int_{\frac{t}{2}}^t e^{-(t-s)\partial^3_x}\partial_x u^5(s) ds \right\|_{L^{\infty}_{x}} \notag \\ 
\lesssim & \int_{\frac{t}{2}}^t|t-s|^{-\frac{1}{2}}\big\|u^4 u_x\big\|_{L^{1}_{x}}ds \notag \\
\lesssim &  \int_{\frac{t}{2}}^t|t-s|^{-\frac{3}{8}}\|u\|_{L^{2}_{x}}|t-s|^{-\frac{1}{8}}\big\|u^3 u_x\big\|_{L^{2}_{x}}ds \notag \\ 
\lesssim &   |t|^{\frac{1}{8}}\|u\|_{L^{\infty}_{t}L^{2}_{x}} \left(\int_{\frac{t}{2}}^t |t-s|^{-\frac{1}{4}}\big\|u^3 u_x\big\|^2_{L^{2}_{x}}ds\right)^{1/2} \notag \\
\lesssim &   |t|^{-\frac{5}{24}}\|u\|_{L^{\infty}_{t}L^{2}_{x}}\|u \|_{X(T)} \|u\|^2_{L^{\frac{24}{5}}_{x}L^{\infty}_{t}}\|u_x\|_{L^{12}_{x}L^{3}_{t}} \left(\int_{\frac{t}{2}}^t |t-s|^{-\frac{3}{4}}ds\right)^{1/6} \notag \\
\lesssim &   |t|^{-\frac{1}{6}}\|u\|_{L^{\infty}_{t}L^{2}_{x}}\|u \|_{X(T)} \|u\|^2_{L^{\frac{24}{5}}_{x}L^{\infty}_{t}}\|u_x\|_{L^{12}_{x}L^{3}_{t}} \notag \\
\leq& C(\|u_0 \|_{H^{\frac{5}{12}}})t^{-\frac1 6}\|u \|_{X(T)}.  \label{DispDecay0Ia2-deri}
	\end{align}
Validity of the last step in  \eqref{DispDecay0Ia2-deri} is from Lemma \ref{gKdVLNorm L245xLqqt} and \eqref{gKdVL(XT)Norm3}.

Then, it follows from \eqref{DispDecay0Ia1-deri} and \eqref{DispDecay0Ia2-deri} that
\begin{align}
\sup_{t\in [0, T]}|t|^{\frac{1}{6}}\|D^{\frac{1}{2}}_xu(t)  \|_{L^{\infty}_{x}} \leq  \|u_0 \|_{L^{1}_{x}}+C(\|u_0\|_{H^{\frac{5}{12}}})\|u \|_{X(T)}. \label{DispDecay0mgKdVf2}
	\end{align}
Combining \eqref{DispDecay0mgKdVf1} and \eqref{DispDecay0mgKdVf2}, we have 
\begin{align}
\|u \|_{X(T)} \leq  2\|u_0 \|_{L^{1}_{x}}+C(\|u_0\|_{H^{\frac{1}{2}}})\left(\|u \|_{X(T)}+\|u \|^2_{X(T)}+\|u \|^3_{X(T)}\right). \nonumber
	\end{align}
As $\|u_0\|_{H^{\frac{1}{2}}}$ is small enough, a simple continuity argument then yields the desired result \eqref{MainResult1a}. We complete the proof of this theorem. \hspace{66.2mm}$\square$

\vspace{1mm} In the next place, we show Theorem \ref{MainResult2}.

{\bf{Proof of Theorem \ref{MainResult2}.}} We only need to estimate the nonlinear term. Denote
$$\|u \|_{X(T)}:=\sup_{t\in (0,T]}|t|^{\frac{1}{3}}\| u(t)  \|_{L^{\infty}_{x}},$$ 
and we split the nonlinear part into $\int_0^t=\int_{0}^{\frac{t}{2}}+\int_{\frac{t}{2}}^t$.

For the first part we proceed as in \eqref{DispDecay0Ia1} to derive
\begin{align}
&\left\| \int_{0}^{\frac{t}{2}}e^{-(t-s)\partial^3_x}\partial_x u^{k+1}(s) ds \right\|_{L^{\infty}_{x}} \notag \\
\lesssim & \int_{0}^{\frac{t}{2}}|t-s|^{-\frac{1}{2}}\big\|D^{\frac{1}{2}}_x u^{k+1}(s) \big\|_{L^{1}_{x}}ds \notag \\
\lesssim &  \int_{0}^{\frac{t}{2}}|t-s|^{-\frac{1}{2}}\big\|u^kD^{\frac{1}{2}}_x u(s) \big\|_{L^{1}_{x}}ds+\int_{0}^{\frac{t}{2}}|t-s|^{-\frac{1}{2}}\|u\|_{L^{2}_{x}}\big\|D^{\frac{1}{2}}_xu^k\big\|_{L^{2}_{x}}ds.   \label{DispDecay0Ia1-k6}
	\end{align}
Thus, from  \eqref{KPV-gKdV-Global-kg4-A1} and \eqref{KPV-gKdV-Global-kg4-A2} it is easy to see 
\begin{align}
\int_{0}^{\frac{t}{2}}|t-s|^{-\frac{1}{2}}\big\|u^kD^{\frac{1}{2}}_x u \big\|_{L^{1}_{x}}ds  
\lesssim&  |t|^{-\frac{1}{2}} \int_{0}^{\frac{t}{2}}\big\|u^2D^{\frac{1}{2}}_x u \big\|_{L^{2}_{x}}\|u \|_{L^{2}_{x}}\|u \|^{k-3}_{L^{\infty}_{x}}ds  \notag \\
\lesssim& |t|^{-\frac{1}{2}} \|u \|_{X(T)}\|u \|^{k-4}_{L^{\infty}_{t}H^{\frac{1}{2}}_{x}}\int_{0}^{\frac{t}{2}}s^{-\frac{1}{3}}\big\|u^2D^{\frac{1}{2}}_x u \big\|_{L^{2}_{x}}\|u \|_{L^{2}_{x}}ds  \notag \\
\lesssim& |t|^{-\frac{1}{2}}\|u \|_{X(T)} \|u \|^{k-3}_{L^{\infty}_{t}H^{\frac{1}{2}}_{x}}\big\|u^2D^{\frac{1}{2}}_x u \big\|_{L^{2}_{tx}} \left(\int_{0}^{\frac{t}{2}}s^{-\frac{2}{3}}ds\right)^{\frac{1}{2}} \notag \\
\lesssim&	|t|^{-\frac{1}{3}}\|u \|_{X(T)} \|u \|^{k-3}_{L^{\infty}_{t}H^{\frac{1}{2}}_{x}}\| u\|^2_{L_{x}^{5} L_{t}^{10}}\big\|D_x^{\frac{1}{2}} u\big\|_{L_{x}^{10} L_{t}^{\frac{10}{3}}} \notag \\
\leq&	C(\|u_0\|_{H^{\frac{1}{2}}})|t|^{-\frac{1}{3}}\|u \|_{X(T)}. \label{DispDecay0Ia2-k6}
	\end{align}
Moreover, we obtain
\begin{align}
 &\int_{0}^{\frac{t}{2}}|t-s|^{-\frac{1}{2}}\|u\|_{L^{2}_{x}}\big\|D^{\frac{1}{2}}_xu^k\big\|_{L^{2}_{x}}ds \notag \\
 \lesssim  &|t|^{-\frac{1}{2}}\int_{0}^{\frac{t}{2}}\|u\|_{L^{2}_{x}}\|u^k\|_{L^{2}_{x}}ds+|t|^{-\frac{1}{2}}\int_{0}^{\frac{t}{2}}\|u\|_{L^{2}_{x}}\big\|D_xu^k\big\|_{L^{2}_{x}}ds  \notag \\
 \lesssim  &|t|^{-\frac{1}{2}}\|u \|_{X(T)} \|u \|^{k-3}_{L^{\infty}_{t}H^{\frac{1}{2}}_{x}}\|u\|^3_{L^{6}_{t}L^{\infty}_{x}} \left(\int_{0}^{\frac{t}{2}}s^{-\frac{2}{3}}ds\right)^{\frac{1}{2}}\notag \\
&\hspace{3mm}+|t|^{-\frac{1}{2}}\|u \|_{X(T)} \|u \|^{k-3}_{L^{\infty}_{t}H^{\frac{1}{2}}_{x}}\big\|u^2u_x\big\|_{L^{2}_{tx}} \left(\int_{0}^{\frac{t}{2}}s^{-\frac{2}{3}}ds\right)^{\frac{1}{2}} \notag \\
\lesssim  &|t|^{-\frac{1}{3}}\|u \|_{X(T)} \|u \|^{k-3}_{L^{\infty}_{t}H^{\frac{1}{2}}_{x}}\left(\|u\|^3_{L^{6}_{t}L^{\infty}_{x}} +\|u \|^2_{L^{4}_{x}L^{\infty}_{t}}\|u_x\|_{L^{\infty}_{x}L^{2}_{t}}\right) \notag \\
\leq&	C(\|u_0\|_{H^{\frac{1}{2}}})|t|^{-\frac{1}{3}}\|u \|_{X(T)}. \label{DispDecay0Ia3-k6}
	\end{align}
via Gagliardo-Nirenberg inequality. Collecting  \eqref{DispDecay0Ia1-k6}--\eqref{DispDecay0Ia3-k6} yields that
\begin{align}
\left\| \int_{0}^{\frac{t}{2}}e^{-(t-s)\partial^3_x}\partial_x u^{k+5}(s) ds \right\|_{L^{\infty}_{x}} 
\leq C(\|u_0\|_{H^{\frac{1}{2}}})|t|^{-\frac{1}{3}}\|u \|_{X(T)}.   \label{DispDecay0Ia4-k6}
	\end{align}

For the second part, using  dispersive inequality \eqref{PrelDisDecay1} with  $\alpha=0$, $r=\infty$ and commutator estimate \eqref{lem:Leibnizpoly2},  one has
\begin{align}
\left\| \int_{\frac{t}{2}}^{t} e^{-(t-s)\partial^3_x}\partial_x u^{k+1}(s) ds \right\|_{L^{\infty}_{x}}
\lesssim& \int_{\frac{t}{2}}^{t}|t-s|^{-\frac{1}{3}}\| u^k u_x\|_{L^{1}_{x}}ds \notag \\
\lesssim &  \int_{\frac{t}{2}}^{t}|t-s|^{-\frac{1}{3}}\| u\|_{L^{2}_{x}}\big\|u^{k-1}u_x\big\|_{L^{2}_{x}}ds \notag \\
\lesssim &  |t|^{-\frac{k-3}{3}}\|u \|^{k-3}_{X(T)}\| u\|_{L^{\infty}_{t}L^{2}_{x}} \big\|u^{2}u_x\big\|_{L^{2}_{xt}} \left(\int_{\frac{t}{2}}^{t}|t-s|^{-\frac{2}{3}}ds\right)^{1/2} 
\notag \\
\lesssim &  |t|^{\frac{1}{6}-\frac{k-3}{3}}\|u \|^{k-3}_{X(T)}\| u\|_{L^{\infty}_{t}L^{2}_{x}} \|u \|^2_{L^{4}_{x}L^{\infty}_{t}}\|u_x\|_{L^{\infty}_{x}L^{2}_{t}}\notag \\
\leq&	C(\|u_0\|_{H^{\frac{1}{2}}})|t|^{-\frac{1}{3}}\|u \|^{k-3}_{X(T)}. \label{DispDecay0IIb0-k6}
	\end{align}

Then, \eqref{DispDecay0Ia4-k6} and \eqref{DispDecay0IIb0-k6} immediately yield the desired estimate \eqref{MainResult2a}. This completes the proof. \hspace{128.5mm}$\square$

 \section{Dispersive decay for gZK in 2D}\label{ZK-DispDecay}                      

The IVP of the gZK equation
\begin{equation}
	\left\{
	\begin{aligned}
		&\partial_{t}u +\partial_{x}\Delta u + \partial_{x} u^{k+1} = 0, \\
		&u(0,x,\mathbf{y})=u_0(x,\mathbf{y}),\ \ \ (x, \mathbf{y})\in \mathbb{R}\times \mathbb{R}^{d-1}, \ t\in\mathbb{R}. \label{gZK} \\
	\end{aligned}
	\right.
\end{equation}
is considered, where $d\geq 2$, $k\in \mathbb{N}$, $\mathbf{y}=(y_1, y_2, \cdots, y_{d-1})$ and $\Delta=\partial_{x}^2+\partial_{y_1}^2+\cdots+\partial_{y_{d-1}}^2$ is the
Laplacian. When $k=1$, \eqref{gZK} is called the Zakharov-Kuznetsov equation introduced by Zakharov and Kuznetsov \cite{ZK74} as a model to describe the propagation of ion-sound waves in magnetic fields. The gZK equation \eqref{gZK} is a multi-dimensional extension of the gKdV equation \eqref{gKdV2}.

\eqref{gZK}  possesses the scaling symmetry, namely, if $u(t,x,\mathbf{y})$
solves \eqref{gZK}  with initial data $u_0(x,\mathbf{y})$, then
$$u_{\lambda}(t,x,\mathbf{y})=\lambda^{\frac{2}{k}}u(\lambda^3t,\lambda x,\lambda\mathbf{y}),  \hspace{10mm} \lambda>0$$
also solves \eqref{gZK} initial data $\lambda^{\frac{2}{k}}u_0(\lambda x,\lambda\mathbf{y})$. It follows from 
$$\|u_{\lambda}(t,x,\mathbf{y})\|_{\dot{H}^{s_c}(\mathbb{R}^{d})}=\|u(t,x,\mathbf{y})\|_{\dot{H}^{s_c}(\mathbb{R}^{d})}$$
that $s_c=\frac{d}{2}-\frac{2}{k}$. Hence, we call $\dot{H}^{s_c}$ the critical Sobolev space of  \eqref{gZK}. This equation has the conserved mass and energy:
\begin{align}
M(u(t))=\int_{\mathbb{R}^{d}}u^2(t)dxd\mathbf{y}=M(u_0), \nonumber
	\end{align}
\begin{align}
E(u(t))=\frac{1}{2}\int_{\mathbb{R}^{d}}|\nabla u (t)|^2dxd\mathbf{y}-\frac{1}{k+2}\int_{\mathbb{R}^{d}}u^{k+2}(t)dxd\mathbf{y}=E(u_0). \nonumber
	\end{align}

There are lots of works on the gZK equation \eqref{gZK}. In two-dimensional case, for $k\geq 3$ we refer to the papers \cite{FLP12, Gru15, LP11, RV12} on well-posedness.  Linares and Pastor \cite{LP11} proved that \eqref{gZK} is small data globally well-posed in $H^1(\mathbb{R}^{2})$ for $k\geq 3$ by using conservation laws mentioned above. Farah, Linares and Pastor  \cite{FLP12} considered  the large time behavior of the solution to \eqref{gZK} with $k\geq 3$. They obtained the decay and scattering results. Specifically, they showed that if $u_0\in H^1(\mathbb{R}^{2})\cap L^{\frac{2(k+1)}{2k+1}}(\mathbb{R}^{2})$ and  
$$\|u_0\|_{L^{\frac{2(k+1)}{2k+1}}}+\|u_0\|_{H^{1}}\ll 1,$$ 
then the global solution $u(t)$ to \eqref{gZK} satisfies
\begin{align}
\sup_{t\neq 0}(1+|t|)^{\frac{2k}{3(k+1)}}\| u(t)  \|_{L^{2(k+1)}_{xy}} \leq C. \nonumber
	\end{align}
Moreover, there exist $u_0^{\pm}\in H^1(\mathbb{R}^{2})$ such that
$$\lim_{t\to \pm \infty} \big\|u(t)-U(t)u_0^{\pm}\big\|_{H^1}=0.$$

By using the global well-posedness result of \cite{FLP12} and Lorentz-Strichartz estimates in Lemma \ref{2dZK-PrelStrichartz}, we get space-time bounds in mixed Lorentz spaces for solutions to the 2D gZK equation \eqref{gZK}. 

\begin{lemma}[Lorentz spacetime bounds] \label{2dZK-LorStri-sol}
Let $4< r \leq \infty$, $\frac{1}{r}+\frac{1}{r'}=1$, $k\in \mathbb{N}$. Assume that $u(t)$ is the global strong solution to \eqref{gZK} with small initial data $u_0$  satisfying
\begin{equation}
	\left\{
	\begin{aligned}
		&u_0\in H^{1}\cap L^{r'}, \hspace{3mm}  \|u_0\|_{H^{1}}\ll1,    \hspace{5mm}  \text{if} \hspace{3mm}  k=3,\\
		&u_0\in H^{2}\cap L^{r'}, \hspace{3mm} \|u_0\|_{H^{2}}\ll1,    \hspace{5mm}   \text{if} \hspace{3mm}  k\geq 4, \nonumber \\ 
	\end{aligned}
	\right.
\end{equation}
 then there exists a constant $C$ dependent on $u_0$, such that 
\begin{align}
\| u(t,x,y)\|_{L^{\frac{6r}{r+4},2}_{t}L^{\frac{4r}{r-4}}_{xy}} \leq C(\|u_0\|_{H^{1}_{xy}}),\hspace{5mm}  \text{if} \hspace{3mm}  k=3, \label{2dZK-LorStri-sol01} \\
\| u(t,x,y)\|_{L^{\frac{3(k-1)r}{r+4},2}_{t}L^{\frac{2(k-1)r}{r-4}}_{xy}} \leq C(\|u_0\|_{H^{2}_{xy}}),\hspace{5mm}  \text{if} \hspace{3mm}  k\geq 4. \label{2dZK-LorStri-sol02}
\end{align}
\end{lemma}
\begin{proof} Applying the Duhamel formula
$$u(t)=U(t)u_0+\int_0^t U(t-s)u^3\partial_x u(s) ds$$
together with Lorentz-Strichartz estimates \eqref{2dZK-Stri1}, \eqref{2dZK-Stri2} and H\"older inequality, we get
\begin{align}
\|u\|_{L^{3,2}_{t}L^{\infty}_{xy}}&\lesssim \|U(t)u_0\|_{L^{3,2}_{t}L^{\infty}_{xy}}+\left\|\int_0^t U(t-s)u^3 \partial_x u(s) ds\right\|_{L^{3,2}_{t}L^{\infty}_{xy}}\notag \\
&\lesssim \|u_0\|_{L^2_{xy}}+ \left\|u^3 \partial_x u\right\|_{L^{1,2}_{t}L^{2}_{xy}}  \notag \\
&\lesssim \|u_0\|_{L^2_{xy}}+ \left\|u\right\|^3_{L^{3,2}_{t}L^{\infty}_{xy}}
 \left\| \partial_x u\right\|_{L^{\infty}_{t}L^{2}_{xy}}.   \nonumber
	\end{align}

Note that 
$$ \left\| \partial_x u\right\|_{L^{\infty}_{t}L^{2}_{xy}}\leq \|u\|_{L^{\infty}_{t}H^{1}_{xy}}\lesssim \|u_0\|_{L^{\infty}_{t}H^{1}_{xy}}\ll 1,$$
then a standard bootstrap argument deduces
$$\|u\|_{L^{3,2}_{t}L^{\infty}_{xy}} \leq C(\|u_0\|_{H^{1}_{xy}}).$$

Using Lorentz-Strichartz estimates \eqref{2dZK-Stri1} and \eqref{2dZK-Stri2} again, one has
\begin{align}
\|u\|_{L^{\frac{6r}{r+4},2}_{t}L^{\frac{4r}{r-4}}_{xy}}&\lesssim \|U(t)u_0\|_{L^{\frac{6r}{r+4},2}_{t}L^{\frac{4r}{r-4}}_{xy}}+\left\|\int_0^t U(t-s)u^3 \partial_x u(s) ds\right\|_{L^{\frac{6r}{r+4},2}_{t}L^{\frac{4r}{r-4}}_{xy}}\notag \\
&\lesssim \|u_0\|_{L^2_{xy}}+ \left\|u^3 \partial_x u\right\|_{L^{1,2}_{t}L^{2}_{xy}}  \notag \\
&\lesssim \|u_0\|_{L^2_{xy}}+ \left\|u\right\|^3_{L^{3,2}_{t}L^{\infty}_{xy}}
 \left\| \partial_x u\right\|_{L^{\infty}_{t}L^{2}_{xy}}<\infty   \nonumber
	\end{align}
which gives the desired estimate \eqref{2dZK-LorStri-sol01}.

Similarly, if $k\geq 4$ then by Sobolev embedding inequality and Lorentz-Strichartz estimates we have
\begin{align}
\|u\|_{L^{\frac{3(k-1)r}{r+4},2}_{t}L^{\frac{2(k-1)r}{r-4}}_{xy}}&\lesssim \|J^{\sigma_k}u\|_{L^{\frac{3(k-1)r}{r+4},2}_{t}L^{\frac{2(k-1)r}{(k-2)r-4}}_{xy}}\notag \\
&\lesssim \|J^{\sigma_k}U(t)u_0\|_{L^{\frac{3(k-1)r}{r+4},2}_{t}L^{\frac{2(k-1)r}{(k-2)r-4}}_{xy}}\notag \\
&\hspace{5mm}+\left\| J^{\sigma_k}\int_0^t U(t-s)\partial_xu^{k+1}(s) ds\right\|_{L^{\frac{3(k-1)r}{r+4},2}_{t}L^{\frac{2(k-1)r}{(k-2)r-4}}_{xy}}\notag \\
&\lesssim \|u_0\|_{H^{\sigma_k}_{xy}}+ \left\|J^{\sigma_k}u^{k}u_x \right\|_{L^{1,2}_{t}L^{2}_{xy}}   \label{2dZK-LorStri-sol03}
	\end{align}
where $0<\sigma_k=1-\frac{2}{k-1}<1$.

Using Kato-Ponce inequality \eqref{KPlem:Leibniz-1} yields
\begin{align}
\left\|J^{\sigma_k}u^{k}u_x\right\|_{L^{2}_{xy}} \lesssim& \left\|u^{k}J^{\sigma_k}u_x\right\|_{L^{2}_{xy}}+\|\partial u^{k}\|_{L^{\infty}_{xy}}\left\|J^{\sigma_k}u\right\|_{L^{2}_{xy}}+\left\|J^{\sigma_k}u^{k}\right\|_{L^{2}_{xy}}\|u_x\|_{L^{\infty}_{xy}} \notag \\
 \lesssim& \|u\|^k_{L^{\infty}_{xy}}\left\|J^{\sigma_k}u_x\right\|_{L^{2}_{xy}}+\|\partial u\|_{L^{\infty}_{xy}}\|u\|^{k-1}_{L^{\infty}_{xy}}\left\|J^{\sigma_k}u\right\|_{L^{2}_{xy}} \notag \\
&+\big\|u^{k}\big\|_{L^{2}_{xy}}\|u_x\|_{L^{\infty}_{xy}}+\big\|\partial u^{k}\big\|_{L^{2}_{xy}}\|u_x\|_{L^{\infty}_{xy}}
\notag \\
 \lesssim& \|u\|^k_{L^{\infty}_{xy}}\|u\|_{H^{2}_{xy}}+\|u\|^{k-1}_{L^{\infty}_{xy}}\|u\|^2_{H^{2}_{xy}} \nonumber
	\end{align}
which further gives by global well-posedness that
\begin{align}
\left\|J^{\sigma_k}u^{k}u_x\right\|_{L^{1,2}_{t}L^{2}_{xy}} \lesssim &\|u\|^3_{L^{3}_{t}L^{\infty}_{xy}}\|u\|^{k-3}_{L^{\infty}_{txy}}\|u\|_{L^{\infty}_{t}H^{2}_{xy}}+\|u\|^{3}_{L^{3}_{t}L^{\infty}_{xy}}\|u\|^{k-4}_{L^{3}_{t}L^{\infty}_{xy}}\|u\|^2_{L^{\infty}_{t}H^{2}_{xy}} \notag \\
\lesssim &\|u\|^3_{L^{3}_{t}L^{\infty}_{xy}}\|u\|^{k-2}_{L^{\infty}_{t}H^{2}_{xy}}+\|u\|^{3}_{L^{3}_{t}L^{\infty}_{xy}}\|u\|^{k-2}_{L^{\infty}_{t}H^{2}_{xy}}<\infty . \label{2dZK-LorStri-sol04}
	\end{align}
Then, \eqref{2dZK-LorStri-sol02} follows from \eqref{2dZK-LorStri-sol03} and \eqref{2dZK-LorStri-sol04} immediately. This completes the proof of the lemma.
\end{proof}

With these estimates in hands, now let us turn to show Theorem \ref{MainResult3}.

{\bf{Proof of Theorem \ref{MainResult3}.}}
Denote
$$\|u \|_{X(T)}:=\sup_{t\in (0, T]}|t|^{\frac{2}{3}(1-\frac{2}{r})}\| u(t)  \|_{L^{r}_{xy}}.$$
We write down the Duhamel formula of $u$ and decompose the region of integration into $[0, \frac{t}{2}]$ and $[\frac{t}{2}, t]$ 
\begin{align}
u(t)&=U(t)u_0+\int_0^t U(t-s)u^k\partial_x u(s) ds  \notag \\
&=U(t)u_0+\int_0^{\frac{t}{2}} U(t-s)u^k\partial_x u(s) ds+\int_{\frac{t}{2}}^t  U(t-s)u^k\partial_x u(s) ds.  \nonumber
\end{align}

Firstly, we consider the case $k=3$. By dispersive estimate \eqref{2dZK-PrelDisDecay1}, the contribution of linear term  is easily seen to be acceptable:
$$\|U(t)u_0\|_{L^r_{xy}}\lesssim t^{-\frac{2}{3}(1-\frac{2}{r})}\|u_0\|_{L^{r'}_{xy}}.$$

To estimate the contribution of nonlinear term, we use dispersive estimate \eqref{2dZK-PrelDisDecay1}, Lorentz-Strichartz estimates and H\"older inequality in Lorentz spaces 
\begin{align}
&\left\| \int_0^{\frac{t}{2}}U(t-s) u^3 \partial_x u(s) ds \right\|_{L^{r}_{xy}} \notag \\ 
\lesssim& \int_{0}^{\frac{t}{2}}|t-s|^{-\frac{2}{3}(1-\frac{2}{r})}\big\|u^3 \partial_x u(s) \big\|_{L^{r'}_{xy}}ds\notag \\ 
 \lesssim& |t|^{-\frac{2}{3}(1-\frac{2}{r})} \int_{0}^{\frac{t}{2}}|s|^{-\frac{2}{3}(1-\frac{2}{r})} \|u \|_{X(T)}\|u_x\|_{L^{2}_{xy}} \|u  \|^{2}_{L^{\frac{4r}{r-4}}_{xy}}ds\notag \\ 
\lesssim&  |t|^{-\frac{2}{3}(1-\frac{2}{r})}\|u \|_{X(T)} \|u_x \|_{L^{\infty}_{t}L^{2}_{xy}}\|u\|^{2}_{L^{\frac{6r}{r+4},2}_{t}L^{\frac{4r}{r-4}}_{xy}}\big\||s|^{-\frac{2}{3}(1-\frac{2}{r})}\big\|_{L^{\frac{3r}{2(r-2)}, \infty}_s}\notag \\ 
\lesssim& C\big(\|u_0\|_{H^{1}_{xy}}\big) |t|^{-\frac{2}{3}(1-\frac{2}{r})}\|u \|_{X(T)},   \label{2dZK-DispDecay01}
	\end{align}
and
\begin{align}
&\left\| \int_{\frac{t}{2}}^t U(t-s) u^3 \partial_x u(s) ds \right\|_{L^{r}_{xy}} \notag \\ 
\lesssim& \int_{\frac{t}{2}}^t |t-s|^{-\frac{2}{3}(1-\frac{2}{r})}\big\|u^3 \partial_x u(s) \big\|_{L^{r'}_{xy}}ds\notag \\ 
 \lesssim& |t|^{-\frac{2}{3}(1-\frac{2}{r})}\|u \|_{X(T)} \int_{\frac{t}{2}}^t |t-s|^{-\frac{2}{3}(1-\frac{2}{r})} \|u_x\|_{L^{2}_{xy}} \|u  \|^{2}_{L^{\frac{4r}{r-4}}_{xy}}ds\notag \\ 
\lesssim&  |t|^{-\frac{2}{3}(1-\frac{2}{r})}\|u \|_{X(T)} \|u_x \|_{L^{\infty}_{t}L^{2}_{xy}}\|u\|^{2}_{L^{\frac{6r}{r+4},2}_{t}L^{\frac{4r}{r-4}}_{xy}}\big\||s|^{-\frac{2}{3}(1-\frac{2}{r})}\big\|_{L^{\frac{3r}{2(r-2)}, \infty}_s}\notag \\ 
\leq& C\big(\|u_0\|_{H^{1}_{xy}}\big) |t|^{-\frac{2}{3}(1-\frac{2}{r})}\|u \|_{X(T)}.   \label{2dZK-DispDecay02}
	\end{align}

The mixed Lorentz space improvement in  \eqref{2dZK-LorStri-sol01} is important to compensate for the fact that $|s|^{-\frac{2}{3}(1-\frac{2}{r})}$ is not in $L^{\frac{3r}{2(r-2)}}_s$ but lies in the Lorentz space $L^{\frac{3r}{2(r-2)}, \infty}_s$.

Secondly, for the case $k\geq 4$, proceeding directly as above yields
\begin{align}
&\left\|  \int_0^{\frac{t}{2}}U(t-s) u^k \partial_x u(s) ds \right\|_{L^{r}_{xy}} \notag \\ 
\lesssim&  \int_0^{\frac{t}{2}}|t-s|^{-\frac{2}{3}(1-\frac{2}{r})}\big\|u^k \partial_x u(s) \big\|_{L^{r'}_{xy}}ds\notag \\ 
 \lesssim& |t|^{-\frac{2}{3}(1-\frac{2}{r})}  \int_0^{\frac{t}{2}}|s|^{-\frac{2}{3}(1-\frac{2}{r})} \|u \|_{X(T)}\|u_x\|_{L^{2}_{xy}} \|u  \|^{k-1}_{L^{\frac{2(k-1)r}{r-4}}_{xy}}ds\notag \\ 
\lesssim&  |t|^{-\frac{2}{3}(1-\frac{2}{r})}\|u \|_{X(T)} \|u_x \|_{L^{\infty}_{t}L^{2}_{xy}}\|u\|^{k-1}_{L^{\frac{3(k-1)r}{r+4},2}_{t}L^{\frac{2(k-1)r}{r-4}}_{xy}}\big\||s|^{-\frac{2}{3}(1-\frac{2}{r})}\big\|_{L^{\frac{3r}{2(r-2)}, \infty}_s}\notag \\ 
\leq& C\big(\|u_0\|_{H^{2}_{xy}}\big) |t|^{-\frac{2}{3}(1-\frac{2}{r})}\|u \|_{X(T)}.   \nonumber
	\end{align}
The other part can be controlled by a similar way. Hence, we finish the proof.\hspace{13.8mm}$\square$

 \section{Energy-(sub)critical gZK in high dimensions}\label{hZK-DispDecay}
We consider three dimensional energy-critical ($s_c=d/2-2/k=1$ with $d=3, k=4$) gZK  and   
four dimensional energy-subcritical ($s_c>1$ with $d=4, k=3$) gZK  equation in this section. 

 Dispersive estimate for the linear ZK operator $e^{-t\partial_x\Delta}$ in higher dimensional ($d\geq 3$) is slightly different from that in two dimensional situation.

 
\begin{lemma}[Proposition 14 in \cite{Schippa20}] \label{hBOZK-PrelDisDecay}
Let $a\geq 1$, $d\geq 3$ and $\psi: \mathbb{R}^d\to \mathbb{R}$  be a smooth radial function
supported in $B_{d}(0,2)\setminus B_{d}(0,1/2)$. Then, we find the following estimate to hold
	\begin{align}
\left|\int_{\mathbb{R}^d} \psi(|\xi|)e^{i(t\xi_1|\xi|^a+x\cdot \xi)} d\xi \right|\leq C|t|^{-1} \label{hBOZK-PrelDisDecay1}
	\end{align}
with $C$ only depending on $d$, $\psi$ and $a$.
\end{lemma}

Interpolating dispersive estimate \eqref{hBOZK-PrelDisDecay1} (by taking $a=2$ for ZK) and conservation of mass implies 
	\begin{align}
\left\| U(t)P_1 u_0 \right\|_{L^r(\mathbb{R}^d)}\lesssim |t|^{-(1-\frac{2}{r})} \left\| \widetilde{P}_1 u_0 \right\|_{L^{r'}(\mathbb{R}^d)} \label{hBOZK-PrelDisDecay001}
	\end{align}
for $2\leq r \leq \infty$ and $d\geq 3$,  where  $1/ r +1/r'=1$ and $\widetilde{P}_N$ is defined in a similar way as $P_N$ but with the cut-off function equal to one on the support of $\chi$.

A scaling argument gives from  \eqref{hBOZK-PrelDisDecay001} that 
	\begin{align}
\left\| U(t)P_N u_0 \right\|_{L^r(\mathbb{R}^d)}\lesssim |t|^{-(1-\frac{2}{r})} N^{(d-3)(1-\frac{2}{r})}\left\| \widetilde{P}_N u_0 \right\|_{L^{r'}(\mathbb{R}^d)}, \nonumber
	\end{align}
and then by using Littlewood-Paley theory we get
	\begin{align}
\left\| U(t) u_0 \right\|_{L^r(\mathbb{R}^d)}\lesssim |t|^{-(1-\frac{2}{r})} 
\left\| (-\Delta)^{(d-3)(\frac{1}{2}-\frac{1}{r})}u_0 \right\|_{L^{r'}(\mathbb{R}^d)}  \label{hBOZK-PrelDisDecay002}
	\end{align}
for $2\leq r \leq \infty$, $1/ r +1/r'=1$ and $d\geq 3$. In view of \eqref{hBOZK-PrelDisDecay002},  Hardy-Littlewood-Sobolev inequality  help yield the following Lorentz-Strichartz estimates
	\begin{align}
\left\| U(t) u_0 \right\|_{L^{q,2} (\mathbb{R}; L^r(\mathbb{R}^d) )}\lesssim  
\| u_0 \|_{\dot{H}^{s}(\mathbb{R}^d)}  \label{hBOZK-PrelDisDecay003a}
	\end{align}
and
	\begin{align}
\left\|\int_0^t U(t-s) g(s)ds \right\|_{L^{q,2} (\mathbb{R}; L^r(\mathbb{R}^d) )}\lesssim  
\| (-\Delta)^{\frac{1}{2}(s+\tilde{s})}g \|_{L^{\tilde{q}',2} (\mathbb{R}; L^{\tilde{r}'}(\mathbb{R}^d) )}  \label{hBOZK-PrelDisDecay003b}
	\end{align}
where $2< q, r,\tilde{q},\tilde{r} < \infty$, $\frac{2}{q} +\frac{2}{r}=\frac{2}{\tilde{q}} +\frac{2}{\tilde{r}}=1$,  $d\geq 3$ and $s=d(\frac{1}{2}-\frac{1}{r})-\frac{3}{q}$, $\tilde{s}=d(\frac{1}{2}-\frac{1}{\tilde{r}})-\frac{3}{\tilde{q}}$.

\subsection{3D energy-critical gZK}

We are now ready to show dispersive estimate for solutions to the 3D energy-critical gZK equation.  \eqref{gZK} is small data globally well-posed in critical Sobolev space $H^{s_c}(\mathbb{R}^d)$ with $s_c=\frac{d}{2}-\frac{2}{k}$ for $d=2$ and $d=3$, see for example \cite{Gru15}.

{\bf{Proof of Theorem \ref{MainResult3}.}} First of all, let us establish global-in-time  bound in Lorentz-Strichartz norm for solutions the 3D energy-critical gZK equation \eqref{gZK}. By Duhamel's principle  
$$u(t)=U(t)u_0+\int_0^t U(t-s)u^4\partial_x u(s) ds, $$
using Lorentz-Strichartz estimates \eqref{hBOZK-PrelDisDecay003a} and \eqref{hBOZK-PrelDisDecay003b}, one gets
\begin{align}
\|Ju \|_{L^{\frac{3r}{2}, 2 }_tL^{\frac{6r}{3r-4}}_{x\mathbf{y}}}\leq &\|JU(t)u_0 \|_{L^{\frac{3r}{2}, 2 }_tL^{\frac{6r}{3r-4}}_{x\mathbf{y}}} +\left\|J\int_0^t U(t-s)u^4\partial_x u(s) ds\right\|_{L^{\frac{3r}{2}, 2 }_tL^{\frac{6r}{3r-4}}_{x\mathbf{y}}}  \notag \\ 
\lesssim& \|Ju_0 \|_{L^{2}_{x\mathbf{y}}}+\|Ju^4\partial_x u \|_{L^{\frac{4}{3},2}_tL^{\frac{4}{3}}_{x\mathbf{y}}} \notag \\ 
\lesssim& \|u_0 \|_{H^{1}_{x\mathbf{y}}}+\|u^4J\partial_x u \|_{L^{\frac{4}{3},2}_tL^{\frac{4}{3}}_{x\mathbf{y}}}+\big\|[J, u^4]\partial_x u \big\|_{L^{\frac{4}{3},2}_tL^{\frac{4}{3}}_{x\mathbf{y}}}.\label{3DgZKmrProof1}
	\end{align}
Observe that
\begin{align}
\|u^4J\partial_x u \|_{L^{\frac{3}{2},2}_tL^{\frac{6}{5}}_{x\mathbf{y}}}\lesssim  \|u \|^2_{L^{3, 2 }_tL_{x\mathbf{y}}^{6}}\|u \|^2_{L^{\infty}_{tx\mathbf{y}}}\|J\partial_x u \|_{L^{\infty}_tL^{2}_{x\mathbf{y}}}\lesssim  \|u \|^2_{L^{3, 2 }_tL^{6}_{x\mathbf{y}}}\| u \|^3_{L^{\infty}_tH^{2}_{x\mathbf{y}}}
\label{3DgZKmrProof2}
	\end{align}
and
\begin{align}
\|u \|_{L^{3, 2 }_tL^{6}_{x\mathbf{y}}}\leq &\|U(t)u_0 \|_{L^{3, 2 }_tL^{6}_{x\mathbf{y}}} +\left\|\int_0^t U(t-s)u^4\partial_x u(s) ds\right\|_{L^{3, 2 }_tL^{6}_{x\mathbf{y}}}  \notag \\ 
\lesssim& \|u_0 \|_{L^{2}_{x\mathbf{y}}}+\|u^4\partial_x u \|_{L^{\frac{3}{2},2}_tL^{\frac{6}{5}}_{x\mathbf{y}}} \notag \\ 
\lesssim&  \|u_0 \|_{L^{2}_{x\mathbf{y}}}+\|u \|^2_{L^{3, 2 }_tL_{x\mathbf{y}}^{6}}\|u \|^2_{L^{\infty}_{tx\mathbf{y}}}\|\partial_x u \|_{L^{\infty}_tL^{2}_{x\mathbf{y}}}
\notag \\ 
\lesssim&  \|u_0 \|_{L^{2}_{x\mathbf{y}}}+ \|u \|^2_{L^{3, 2 }_tL^{6}_{x\mathbf{y}}}\| u \|^3_{L^{\infty}_tH^{2}_{x\mathbf{y}}}.\label{3DgZKmrProof3}
	\end{align}

Choosing $\| u_0 \|_{H^{2}_{x\mathbf{y}}}\ll1$ such that
$$\| u \|_{L^{\infty}_tH^{2}_{x\mathbf{y}}}\leq 1/10,$$ 
a standard bootstrap argument yields from \eqref{3DgZKmrProof3} that
\begin{align}
\|u \|_{L^{3, 2 }_tL^{6}_{x\mathbf{y}}}\leq C\big(\|u_0 \|_{H^{2}_{x\mathbf{y}}}\big).\label{3DgZKmrProof4}
	\end{align}
So, by \eqref{3DgZKmrProof2} and \eqref{3DgZKmrProof4} we have
\begin{align}
\|u^4J\partial_x u \|_{L^{\frac{3}{2},2}_tL^{\frac{6}{5}}_{x\mathbf{y}}}\leq C\big(\|u_0 \|_{H^{2}_{x\mathbf{y}}}\big).\label{3DgZKmrProof5}
	\end{align}
Using Kato-Ponce commutator estimates, one can control the last term in RHS\eqref{3DgZKmrProof1}. Hence,
\begin{align}
\|Ju \|_{L^{\frac{3r}{2}, 2 }_tL^{\frac{6r}{3r-4}}_{x\mathbf{y}}}\leq C\big(\|u_0 \|_{H^{2}_{x\mathbf{y}}}\big).\label{3DgZKmrProof6}
	\end{align}

Next, we consider dispaersive decay. Denote
$$\|u \|_{X(T)}:=\sup_{t\in (0,T]}|t|^{1-\frac{2}{r}} \| u(t)  \|_{L^{r}(\mathbb{R}^3)},$$
and write
\begin{align}
u(t)=U(t)u_0+\int_0^{\frac{t}{2}} U(t-s)u^4\partial_x u(s) ds+\int_{\frac{t}{2}}^t U(t-s)u^4\partial_x u(s) ds.  \nonumber
\end{align}
By dispersive estimate \eqref{hBOZK-PrelDisDecay002}, it suffice to consider the nonlinear term. Taking use of dispersive estimate \eqref{hBOZK-PrelDisDecay002}, Sobolev embedding inequality and \eqref{3DgZKmrProof6} yields
\begin{align}
&\left\| \int_{0}^{\frac{t}{2}}U(t-s) u^4\partial_{x}u (s) ds \right\|_{L^{r}(\mathbb{R}^d)} \notag \\ 
\lesssim& \int_{0}^{\frac{t}{2}}|t-s|^{-(1-\frac{2}{r})}\big\|u^4\partial_{x}u(s) \big\|_{L^{r'}(\mathbb{R}^d)}ds\notag \\ 
 \lesssim& |t|^{-(1-\frac{2}{r})} \int_{0}^{\frac{t}{2}} \|u \|_{L^{r}(\mathbb{R}^d)} \|u \|^3_{L^{\frac{6r}{r-4}}(\mathbb{R}^d)}
\|\partial_{x}u\|_{L^{2}(\mathbb{R}^d)} ds\notag \\ 
\lesssim& |t|^{-(1-\frac{2}{r})} \|u \|_{X(T)}\|\partial_{x}u\|_{L^{\infty}_tL^{2}_{x\mathbf{y}}} 
 \|u \|^3_{L^{\frac{3r}{2}, 3 }_tL^{\frac{6r}{r-4}}_{x\mathbf{y}}}  \big\||s|^{-(1-\frac{2}{r})}\big\|_{L^{\frac{r}{r-2}, \infty }_s}\notag \\ 
\lesssim& |t|^{-(1-\frac{2}{r})} \|u \|_{X(T)}\|u\|_{L^{\infty}_tH^{1}_{x\mathbf{y}}} 
 \|Ju \|^3_{L^{\frac{3r}{2}, 2 }_tL^{\frac{6r}{3r-4}}_{x\mathbf{y}}} \notag \\
\leq & C\big(\|u_0 \|_{H^{2}_{x\mathbf{y}}}\big) |t|^{-(1-\frac{2}{r})} \|u \|_{X(T)}.
	\label{3DgZKmrProof7}\end{align}
 Arguing similarly, 
\begin{align}
\left\| \int_{\frac{t}{2}}^tU(t-s) u^4\partial_{x}u (s) ds \right\|_{L^{r}(\mathbb{R}^d)} 
\leq  C\big(\|u_0 \|_{H^{2}_{x\mathbf{y}}}\big) |t|^{-(1-\frac{2}{r})} \|u \|_{X(T)}.
	\label{3DgZKmrProof8}\end{align}

So, we have
$$ \|u \|_{X(T)}\leq  \|u_0 \|_{L^{r'}_{x\mathbf{y}}}+C\big(\|u_0 \|_{H^{2}_{x\mathbf{y}}}\big) |t|^{-(1-\frac{2}{r})} \|u \|_{X(T)}$$
which implies \eqref{MainResult3a} by choosing $\|u_0 \|_{H^{2}_{x\mathbf{y}}}\ll 1$. We finish the proof. 
\hspace{31mm}$\square$

\subsection{4D energy-subcritical gZK}

In this subsection, we consider four dimensional energy-subcritical gZK equation \eqref{gZK} with $k=3$. Herr and Kinoshita \cite{HeKi23} showed that \eqref{gZK} is global well-posed in $H^1(\mathbb{R}^4)$ under a smallness condition. 

We may utilize anisotropic Strichartz estimates derived from the $(d-1)$-dimensional Schr\"odinger equation by Herr and Kinoshita \cite{HeKi21} to study dispersive decay of solutions to  \eqref{gZK} when $d=4$ and $k=3$.

For $d\geq 1$, we say $(q,r)$ is Schr\"odinger $d$-admissible if
$$2 \leq q, r \leq \infty, \hspace{2mm} \frac{2}{q}+\frac{d}{r} = \frac{d}{2}, \hspace{2mm}(d,q,r)\neq (2,2,\infty).$$

\begin{lemma}[Theorem 1.2 in \cite{KT98}]  \label{hShr-PrelDisDecay}
Let $d\geq 1$, and $(q,r), (\tilde{q},\tilde{r})$  be Schr\"odinger $d$-admissible. Then, 
	\begin{align}
& \hspace{6mm} \left\|e^{it\Delta}u_0\right\|_{L^r_{x}} \lesssim
	|t|^{-d(\frac{1}{2}-\frac{1}{r})}	\|u_0\|_{L^{r'}_{x}}, \label{hShr-PrelDisDecay0} \\
&	\hspace{12mm} \left\|e^{it\Delta}u_0\right\|_{L_t^qL^r_{x}}\lesssim
		\|u_0\|_{L^2_{x}}, \label{hShr-PrelDisDecay1} \\
&	\left\|\int_{0}^{t} e^{i(t-s)\Delta}g(s, \cdot)ds\right\|_{L_t^qL^r_{x}}\lesssim
		\|g\|_{L^{\tilde{q}'}_tL^{\tilde{r}'}_{x}}. \label{hShr-PrelDisDecay2}
	\end{align}
where $\frac{1}{\tilde{q}}+\frac{1}{\tilde{q}'}=1$ and $\frac{1}{\tilde{r}}+\frac{1}{\tilde{r}'}=1$.
\end{lemma}

The ZK unitary group $e^{-t\partial_x\Delta}$ is closely related to the  Schr\"odinger unitary group $e^{it\Delta}$.
\begin{lemma} \label{4hZK-PrelDisDecay}
Let $d\geq 2$, and $2\leq  r \leq \infty$. Then, we have
	\begin{align}
\left\|D_{x}^{(d-1)(\frac{1}{2}-\frac{1}{r})}U(t)u_0\right\|_{L^r_{\mathbf{y}}L^2_{x}} \lesssim
	|t|^{-(d-1)(\frac{1}{2}-\frac{1}{r})}	\|u_0\|_{L^{r'}_{\mathbf{y}}L^2_{x}}.\label{4hZK-PrelDisDecay1}
	\end{align}
\end{lemma}

	\begin{proof}
Denote $\Delta_{\mathbf{y}}=\sum_{j=1}^{d-1}\partial_{y_j}^2$ and  define
$$V_{\xi_1}(t)f(\mathbf{y}):=\left(e^{-it\xi_1\Delta_{\mathbf{y}}}f\right)(\mathbf{y})$$
for any fixed $\xi \in\mathbb{R}$. Then
$$U(t)u_0=\mathscr{F}^{-1}_{\xi_1}V_{\xi_1}(t)\widehat{u_0}^{\xi_1}(\xi_1, \mathbf{y}).$$

It follows from Plancherel's identity, Minkowski's inequality and \eqref{hShr-PrelDisDecay0} that
\begin{align}
\left\|U(t)u_0\right\|_{L^r_{\mathbf{y}}L^2_{x}}&= \left\|V_{\xi_1}(t)\widehat{u_0}^{\xi_1}(\xi_1, \mathbf{y})\right\|_{L^r_{\mathbf{y}}L^2_{\xi_1}} \notag \\
&\leq\left\|  \left\|e^{-it\xi_1\Delta_{\mathbf{y}}}\widehat{u_0}^{\xi_1}(\xi_1, \mathbf{y})\right\|_{L^{r}_{\mathbf{y}}} \right\|_{L^2_{\xi_1}}\notag \\
&\lesssim
	\left\| |t\xi_1|^{-(d-1)(\frac{1}{2}-\frac{1}{r})}	 \left\|\widehat{u_0}^{\xi_1}(\xi_1,\mathbf{y})\right\|_{L^{r'}_{\mathbf{y}}} \right\|_{L^2_{\xi_1}} \notag \\
&\lesssim |t|^{-(d-1)(\frac{1}{2}-\frac{1}{r})}
 	 \left\|D_{x}^{-(d-1)(\frac{1}{2}-\frac{1}{r})} u_0\right\|_{L^{r'}_{\mathbf{y}}L^2_{x}} \nonumber
	\end{align}
which implies the desired estimate \eqref{4hZK-PrelDisDecay1}.
	\end{proof}

\begin{lemma}[Theorem 2.1 in \cite{HeKi21}]\label{hZK-PrelStrichartz}
	Let $d\geq 2$  and $(q,r), (\tilde{q},\tilde{r})$  be  Schr\"odinger $(d-1)$-admissible. Then, 
	\begin{align}
		\left\|D_{x}^{\frac{1}{q}}U(t)u_0\right\|_{L_t^qL^r_{\mathbf{y}}L^2_{x}} &\lesssim
		\|u_0\|_{L^2_{x\mathbf{y}}}, \label{hZK-Stri1} \\
	\left\|D_{x}^{\frac{1}{\tilde{q}}}\int U(t-s)g(s, \cdot)ds\right\|_{L^2_{x\mathbf{y}}} &\lesssim
		\|g\|_{L^{\tilde{q}'}_tL^{\tilde{r}'}_{\mathbf{y}}L^2_{x}}, \label{hZK-Stri2} \\
	\left\|D_{x}^{\frac{1}{q}+\frac{1}{\tilde{q}}}\int_{0}^{t} U(t-s)g(s, \cdot)ds\right\|_{L_t^qL^r_{\mathbf{y}}L^2_{x}} &\lesssim
		\|g\|_{L^{\tilde{q}'}_tL^{\tilde{r}'}_{\mathbf{y}}L^2_{x}}. \label{hZK-Stri3}
	\end{align}
where $\frac{1}{\tilde{q}}+\frac{1}{\tilde{q}'}=1$ and $\frac{1}{\tilde{r}}+\frac{1}{\tilde{r}'}=1$.
\end{lemma}

{\bf{Proof of Theorem \ref{MainResult5}.}} Denote
$$\|u\|_{X(T)}:=\sup_{t\in (0, T]}|t|\| u(t,x,\mathbf{y})  \|_{L^{\infty}_{x\mathbf{y}}(\mathbb{R}^4)}.$$
Using dispersive estimate \eqref{hBOZK-PrelDisDecay002} and  H\"older inequality, we obtain
\begin{align}
&\left\| \int_{0}^{\frac{t}{2}}U(t-s) u^3\partial_{x}u (s) ds \right\|_{L^{\infty}_{x\mathbf{y}}(\mathbb{R}^4)} \notag \\ 
\lesssim& \int_{0}^{\frac{t}{2}}|t-s|^{-1}\big\|(-\Delta)^{\frac{1}{2}}u^3\partial_{x}u(s) \big\|_{L^{1}_{x\mathbf{y}}(\mathbb{R}^4)}ds\notag \\ 
 \lesssim& |t|^{-1} \int_{0}^{\frac{t}{2}}\big\|u^3(-\Delta)^{\frac{1}{2}}\partial_{x}u(s) \big\|_{L^{1}_{x\mathbf{y}}(\mathbb{R}^4)}ds+|t|^{-1}  \int_{0}^{\frac{t}{2}}\left\|[(-\Delta)^{\frac{1}{2}}, u^3]\partial_{x}u(s) \right\|_{L^{1}_{x\mathbf{y}}(\mathbb{R}^4)}ds \notag \\ 
 \lesssim& |t|^{-1}  \|u \|^3_{L^{3}_{t}L^{6}_{x\mathbf{y}}}  
\big\|(-\Delta)^{\frac{1}{2}}\partial_{x}u\big\|_{L^{\infty}_{t}L^{2}_{x\mathbf{y}}} +|t|^{-1}  \int_{0}^{\frac{t}{2}} \left\|[(-\Delta)^{\frac{1}{2}}, u^3]\partial_{x}u(s) \right\|_{L^{1}_{x\mathbf{y}}}ds\label{MainResult5a1} 
	\end{align}

Applying the analogue of \eqref{hBOZK-PrelDisDecay003a} and \eqref{hBOZK-PrelDisDecay003b} deduces 
\begin{align}
 \|u \|_{L^{3}_{t}L^{6}_{x\mathbf{y}}}\lesssim&  \|U(t)u_0 \|_{L^{3}_{t}L^{6}_{x}}+\left\| \int_{0}^{t}U(t-s) u^3\partial_{x}u (s) ds \right\|_{L^{3}_{t}L^{6}_{x}}  \notag \\ 
\lesssim&  \|u_0 \|_{\dot{H}^{\frac{1}{3}}}+\left\|(-\Delta)^{\frac{1}{3}}u^3\partial_{x}u \right\|_{L^{\frac{3}{2}}_{t}L^{\frac{6}{5}}_{x\mathbf{y}}} \notag \\ 
\lesssim&  \|u_0 \|_{\dot{H}^{\frac{1}{3}}}+\left\|u^3(-\Delta)^{\frac{1}{3}}\partial_{x}u \right\|_{L^{\frac{3}{2}}_{t}L^{\frac{6}{5}}_{x\mathbf{y}}}+\left\|[(-\Delta)^{\frac{1}{3}},u^3]\partial_{x}u \right\|_{L^{\frac{3}{2}}_{t}L^{\frac{6}{5}}_{x\mathbf{y}}} \notag \\ 
\lesssim&  \|u_0 \|_{\dot{H}^{\frac{1}{3}}}+\|  u \|^2_{L^{3}_{t}L^{6}_{x\mathbf{y}}} \|  u \|_{L^{\infty}_{tx\mathbf{y}}} \big\| (-\Delta)^{\frac{1}{3}}\partial_{x}u \big\|_{L^{\infty}_{t}L^{2}_{x\mathbf{y}}} \notag \\ 
\lesssim&  \|u_0 \|_{\dot{H}^{\frac{1}{3}}}+\|  u \|^2_{L^{3}_{t}L^{6}_{x\mathbf{y}}}\| u \|^2_{L^{\infty}_{t}H^{2}_{x\mathbf{y}}} 
\label{MainResult5a2} 
	\end{align}
If $\| u_0 \|_{H^{2}_{x\mathbf{y}}} \ll 1$ such that $\| u \|_{L^{\infty}_{t}H^{2}_{x\mathbf{y}}} <1/10$, then it follows from \eqref{MainResult5a2} that
\begin{align}
 \|u \|_{L^{3}_{t}L^{6}_{x\mathbf{y}}}< C\big(\| u_0 \|_{H^{2}_{x\mathbf{y}}}\big).
\label{MainResult5a3} 
	\end{align}
One can control the final term of RHS\eqref{MainResult5a1} by using commutator estimate and a similar argument as above. Hence, 
\begin{align}
\left\| \int_{0}^{\frac{t}{2}}U(t-s) u^3\partial_{x}u (s) ds \right\|_{L^{\infty}(\mathbb{R}^4)} 
\leq C\big(\| u_0 \|_{H^{2}_{x\mathbf{y}}}\big) |t|^{-1}. \label{MainResult5a4} 
	\end{align}

By using Sobolev inequality and anisotropic Strichartz estimate \eqref{hZK-Stri2}
\begin{align}
&\left\| \int_{\frac{t}{2}}^t U(t-s) u^3\partial_{x}u (s) ds \right\|_{L^{\infty}(\mathbb{R}^4)} \notag \\ 
\lesssim& \left\|\partial_{x} \int_{\frac{t}{2}}^tU(t-s)J^2 u^4  ds \right\|_{L^{2}(\mathbb{R}^4)}
\notag \\ 
\lesssim& \big\|\mathds{1}_{s\in [\frac{t}{2}, t]}D_{x}^{\frac{1}{2}} J^2 u^4   \big\|_{L^{2}_sL^{\frac{6}{5}}_{\mathbf{y}}L^{2}_{x}}\notag \\ 
\lesssim& \big\|\mathds{1}_{s\in [\frac{t}{2}, t]}u^3 D_{x}^{\frac{1}{2}} J^2u    \big\|_{L^{2}_sL^{\frac{6}{5}}_{\mathbf{y}}L^{2}_{x}}
+ \big\|\mathds{1}_{s\in [\frac{t}{2}, t]}[D_{x}^{\frac{1}{2}} J^2,  u^3] u   \big\|_{L^{2}_sL^{\frac{6}{5}}_{\mathbf{y}}L^{2}_{x}}. \label{MainResult5b1} 
	\end{align}
We only consider the first term in RHS\eqref{MainResult5b1}, as the second term can be controlled similarly. It is easy to see that
\begin{align} 
& \big\|\mathds{1}_{s\in [\frac{t}{2}, t]}u^3 D_{x}^{\frac{1}{2}} J^2u    \big\|_{L^{2}_sL^{\frac{6}{5}}_{\mathbf{y}}L^{2}_{x}}\notag \\ 
 \lesssim& |t|^{-1} \|u \|_{X(T)} \| u\|^2_{L^{\infty}_{t} L^{3}_{\mathbf{y}}L^{\infty}_{x}}  \big\|D_{x}^{\frac{1}{2}} J^2 u\big\|_{L^{2}_{t}L^{6}_{\mathbf{y}}L^{2}_{x}}. \label{MainResult5b2} 
	\end{align}
In view of Lemma \ref{hZK-PrelStrichartz}, we get
\begin{align}
\big\|D_{x}^{\frac{1}{2}} J^2 u\big\|_{L^{2}_{t}L^{6}_{\mathbf{y}}L^{2}_{x}}&\lesssim
\big\|D_{x}^{\frac{1}{2}} J^2 U(t)u_0\big\|_{L^{2}_{t}L^{6}_{\mathbf{y}}L^{2}_{x}}+\left\|D_{x}^{\frac{1}{2}} J^2 \int_{0}^{t} U(t-s) u^3\partial_{x}u (s) ds\right\|_{L^{2}_{t}L^{6}_{\mathbf{y}}L^{2}_{x}} \notag\\
&\lesssim \big\|J^2u_0 \big\|_{L^{2}_{x}}+\big\|D^{\frac{1}{2}}_{x}J^2 u^4 \big\|_{L^{2}_{t}L^{\frac{6}{5}}_{\mathbf{y}}L^{2}_{x}}
\notag\\
&\lesssim \big\|J^2u_0 \big\|_{L^{2}_{x}}+\big\|u^3D^{\frac{1}{2}}_{x}J^2 u \big\|_{L^{2}_{t}L^{\frac{6}{5}}_{\mathbf{y}}L^{2}_{x}}+\big\|[D^{\frac{1}{2}}_{x}J^2, u^3] u \big\|_{L^{2}_{t}L^{\frac{6}{5}}_{\mathbf{y}}L^{2}_{x}}
\notag\\
&\lesssim \|u_0 \|_{H^{2}_{x}}+\| u\|^3_{L^{\infty}_{t} L^{\frac{9}{2}}_{\mathbf{y}}L^{\infty}_{x}}  \big\|D_{x}^{\frac{1}{2}} J^2 u\big\|_{L^{2}_{t}L^{6}_{\mathbf{y}}L^{2}_{x}}
\notag\\
&\lesssim \|u_0 \|_{H^{2}_{x}}+\| u\|^3_{L^{\infty}_{t} H^{2}_{x\mathbf{y}}}  \big\|D_{x}^{\frac{1}{2}} J^2 u\big\|_{L^{2}_{t}L^{6}_{\mathbf{y}}L^{2}_{x}}\nonumber
	\end{align}
from which it follows 
\begin{align}
\big\|D_{x}^{\frac{1}{2}} J^2 u\big\|_{L^{2}_{t}L^{6}_{\mathbf{y}}L^{2}_{x}}<
 C\big(\|u_0 \|_{H^{2}_{x}}\big) \label{MainResult5b3} 
	\end{align}
provided that $\|u_0 \|_{H^{2}_{x}}\ll 1$.

Collecting \eqref{MainResult5b1}, \eqref{MainResult5b2} and \eqref{MainResult5b3} yields
\begin{align}
\left\| \int_{\frac{t}{2}}^tU(t-s) u^3\partial_{x}u (s) ds \right\|_{L^{\infty}(\mathbb{R}^4)} 
\leq C\big(\| u_0 \|_{H^{2}_{x\mathbf{y}}}\big) |t|^{-1}\|u \|_{X(T)} . \label{MainResult5b4} 
	\end{align}

Then, from \eqref{hBOZK-PrelDisDecay002}, \eqref{MainResult5a4} and  \eqref{MainResult5b4}, we obtain
\begin{align}
\|u \|_{X(T)}  \leq C\big(\| u_0 \|_{H^{2}_{x\mathbf{y}}}\big) \left(\|(-\Delta)^{\frac{1}{2}} u_0 \|_{L^{1}_{x\mathbf{y}}}+1+\|u \|_{X(T)}\right) . \nonumber
	\end{align}
which implies \eqref{MainResult5AA} by taking $\| u_0 \|_{H^{2}_{x\mathbf{y}}}\ll1$.

Netx, let us turn to prove \eqref{MainResult5BB}. Utilizing anisotropic dispersive estimate \eqref{4hZK-PrelDisDecay1} with  $r=6$, one gets
\begin{align}
\left\| \partial_{x}\int_{0}^{\frac{t}{2}}U(t-s) u^3\partial_{x}u (s) ds \right\|_{L^{6}_{\mathbf{y}}L^2_{x}} 
\lesssim& \int_{0}^{\frac{t}{2}}|t-s|^{-1}\big\|u^3\partial_{x}u(s) \big\|_{L^{\frac{6}{5}}_{\mathbf{y}}L^2_{x}}ds\notag \\ 
 \lesssim& |t|^{-1} \int_{0}^{\frac{t}{2}} \|u \|^3_{L^{\frac{9}{2}}_{\mathbf{y}}L^{\infty}_{x}}\|\partial_{x}u\|_{L^{6}_{\mathbf{y}}L^{2}_{x}} ds\notag \\ 
\lesssim&  |t|^{-1}  \|u \|^3_{L^{6}_{t}L^{\frac{9}{2}}_{\mathbf{y}}L^{\infty}_{x}}\|\partial_{x}u\|_{L^{2}_{t}L^{6}_{\mathbf{y}}L^{2}_{x}} .   \label{MainResult5c1} 
	\end{align}
Arguing similarly as  \eqref{MainResult5b3}, we have 
\begin{align}
\big\|\partial_{x} u\big\|_{L^{2}_{t}L^{6}_{\mathbf{y}}L^{2}_{x}}<
 C\big(\|u_0 \|_{H^{2}_{x}}\big). \label{MainResult5c2} 
	\end{align}
To estimate the contribution of the first term on RHS\eqref{MainResult5c1}, we employ Sobolev inequality and Strichartz estimates (see \eqref{hBOZK-PrelDisDecay003a} and \eqref{hBOZK-PrelDisDecay003b})
\begin{align}
 \|u \|_{L^{6}_{t}L^{\frac{9}{2}}_{\mathbf{y}}L^{\infty}_{x}}\lesssim \|J^{\frac{2}{3}}u \|_{L^{6}_{t}L^{3}_{x\mathbf{y}}} &\lesssim \|J^{\frac{2}{3}}U(t)u_0 \|_{L^{6}_{t}L^{3}_{x\mathbf{y}}}+\left\|J^{\frac{2}{3}}\int_{0}^{t} U(t-s) u^3\partial_{x}u (s) ds \right\|_{L^{6}_{t}L^{3}_{x\mathbf{y}}}\notag\\
&\lesssim \big\|J^{\frac{2}{3}}u_0 \big\|_{H^{\frac{1}{6}}_{x\mathbf{y}}}+\left\|J^{\frac{5}{6}}(u^3\partial_{x}u) \right\|_{L^{1}_{t}L^{2}_{x\mathbf{y}}}\notag\\
&\lesssim \|u_0 \|_{H^{1}_{x\mathbf{y}}}+\big\|u^3J^{\frac{5}{6}}\partial_{x}u \big\|_{L^{1}_{t}L^{2}_{x\mathbf{y}}}+\left\|[J^{\frac{5}{6}},u^3] \partial_{x}u \right\|_{L^{1}_{t}L^{2}_{x\mathbf{y}}}.  \label{MainResult5c3} 
	\end{align}
Note that
\begin{align}
\big\|u^3J^{\frac{5}{6}}\partial_{x}u \big\|_{L^{1}_{t}L^{2}_{x\mathbf{y}}}\lesssim& \|u \|^3_{L^{3}_{t}L^{\infty}_{x\mathbf{y}}}\big\|J^{\frac{5}{6}}\partial_{x}u \big\|_{L^{\infty}_{t}L^{2}_{x\mathbf{y}}}\lesssim \|J^{\frac{2}{3}}u \|^3_{L^{3}_{t}L^{6}_{x\mathbf{y}}}\|u \|_{L^{\infty}_{t}H^{2}_{x\mathbf{y}}}.  \label{MainResult5c4} 
	\end{align}
Proceeding directly as above yields
\begin{align}
 \|J^{\frac{2}{3}}u \|_{L^{3}_{t}L^{6}_{x\mathbf{y}}} &\lesssim \|J^{\frac{2}{3}}U(t)u_0 \|_{L^{3}_{t}L^{6}_{x\mathbf{y}}}+\left\|J^{\frac{2}{3}}\int_{0}^{t} U(t-s) u^3\partial_{x}u (s) ds \right\|_{L^{3}_{t}L^{6}_{x\mathbf{y}}}\notag\\
&\lesssim \big\|J^{\frac{2}{3}}u_0 \big\|_{H^{\frac{1}{3}}_{x\mathbf{y}}}+\left\|J(u^3\partial_{x}u) \right\|_{L^{1}_{t}L^{2}_{x\mathbf{y}}}\notag\\
&\lesssim \|u_0 \|_{H^{1}_{x\mathbf{y}}}+\|u \|^3_{L^{3}_{t}L^{\infty}_{x\mathbf{y}}}\big\|J\partial_{x}u \big\|_{L^{\infty}_{t}L^{2}_{x\mathbf{y}}}\notag\\
&\lesssim \|u_0 \|_{H^{1}_{x\mathbf{y}}}+\|J^{\frac{2}{3}}u \|^3_{L^{3}_{t}L^{6}_{x\mathbf{y}}}\|u \|_{L^{\infty}_{t}H^{2}_{x\mathbf{y}}}\nonumber  
	\end{align}
which immediately deduces that
\begin{align}
 \|J^{\frac{2}{3}}u \|_{L^{3}_{t}L^{6}_{x\mathbf{y}}}\lesssim \|u_0 \|_{H^{2}_{x\mathbf{y}}}\label{MainResult5c5} 
	\end{align}
as long as $\|u_0\|_{H^{2}_{x}}\ll 1$.

Collecting \eqref{MainResult5c1}-\eqref{MainResult5c5} gives
\begin{align}
\left\| \partial_{x}\int_{0}^{\frac{t}{2}}U(t-s) u^3\partial_{x}u (s) ds \right\|_{L^{6}_{\mathbf{y}}L^2_{x}} 
\lesssim C\big(\|u_0\|_{H^{2}_{x}}\big)|t|^{-1}.   \label{MainResult5c6} 
	\end{align}

To address the remaining term, we combine Sobolev inequality, Strichartz estimates, estimate \eqref{MainResult5AA} and \eqref{MainResult5b3}
\begin{align}
&\left\| \partial_{x}\int_{\frac{t}{2}}^t U(t-s) u^3\partial_{x}u (s) ds \right\|_{L^{6}_{\mathbf{y}}L^2_{x}} \notag \\ 
 \lesssim&\left\| \partial_{x}\int_{\frac{t}{2}}^t U(t-s)J  \big(u^3\partial_{x}u\big)(s) ds \right\|_{L^2_{x\mathbf{y}}} \notag \\ 
\lesssim& \left\| \mathds{1}_{s\in [\frac{t}{2}, t] }D^{\frac{1}{2}}_{x}J \big(u^3\partial_{x}u\big)\right\|_{L^{2}_{s}L^{\frac{6}{5}}_{\mathbf{y}}L^{2}_{x}} \notag \\ 
\lesssim& \left\| \mathds{1}_{s\in [\frac{t}{2}, t] }u^3D^{\frac{1}{2}}_{x}J \partial_{x}u\right\|_{L^{2}_{s}L^{\frac{6}{5}}_{\mathbf{y}}L^{2}_{x}}+\left\| \mathds{1}_{s\in [\frac{t}{2}, t] }[D^{\frac{1}{2}}_{x}J, u^3] \partial_{x}u\right\|_{L^{2}_{s}L^{\frac{6}{5}}_{\mathbf{y}}L^{2}_{x}} \notag \\ 
\lesssim& |t|^{-1} \|u\|^2_{L^{\infty}_{s}L^{3}_{\mathbf{y}}L^{\infty}_{x}}\big\| D^{\frac{1}{2}}_{x}J^2u \big\|_{L^{2}_{t}L^{6}_{\mathbf{y}}L^{2}_{x}}
\lesssim |t|^{-1} \label{MainResult5d1}
	\end{align}

Then, it is  easy to deduce \eqref{MainResult5BB} from dispersive estimate \eqref{4hZK-PrelDisDecay1},   \eqref{MainResult5c6}  and \eqref{MainResult5d1}. We finish the proof of the theorem.\hspace{96.5mm}$\square$

\section*{Acknowledgments}
M.S is partially supported by the NSFC, Grant No. 12101629.

\end{document}